\documentclass[11 pt]{amsart}
\usepackage{amsfonts}
\usepackage{amsmath,amscd}
\usepackage{fullpage}
\usepackage{amssymb}
\usepackage{centernot} 
\usepackage{enumerate} 
\usepackage{parskip}
\usepackage{pb-diagram} 
\usepackage{mathrsfs}
\usepackage[OT2,T1]{fontenc}
\usepackage{seqsplit}
\usepackage{tikz}
\usepackage{tikz-cd}
\usepackage[all,cmtip]{xy}

\usepackage{color}
\usepackage{array}
\usepackage{verbatim}
\usepackage{url}

\def\II{\text{\rm II}}

\def\IV{\text{\rm IV}}

\DeclareSymbolFont{cyrletters}{OT2}{wncyr}{m}{n}
\DeclareMathSymbol{\Sha}{\mathalpha}{cyrletters}{"58}

\definecolor{refkey}{rgb}{1,1,1}
\definecolor{labelkey}{rgb}{1,1,1}
\definecolor{cite}{rgb}{0.9451,0.2706,0.4941}
\definecolor{ruri}{rgb}{0.0078,0.4022,0.8010}

\usepackage[%
bookmarks=true,bookmarksnumbered=true,%
colorlinks=true,linkcolor=ruri,citecolor=red%
 ]{hyperref}

\makeindex 

\def\F{{\rm \mathbb{F}}}
\def\Z{{\rm \mathbb{Z}}}
\def\N{{\rm \mathbb{N}}}
\def\Q{{\rm \mathbb{Q}}}

\def\G{{\rm \mathbb{G}}}
\def\Qbar{{\rm \overline{\mathbb{Q}}}}
\def\C{{\rm \mathbb{C}}}

\def\R{{\rm \mathbb{R}}}

\def\P{{\rm \mathbb{P}}}

\def\A{{\rm \mathbb{A}}}

\def\AJ{{\rm AJ}}

\def\Aut{{\rm Aut}}

\def\Pic{{\rm Pic}}
\def\Gr{{\rm Gr}}

\def\disc{{\rm disc}}

\def\Sym{{\rm Sym}}

\def\GL{{\rm GL}}
\def\Gr{{\rm Gr}}
\def\CH{{\rm CH}}

\DeclareMathOperator*{\ord}{ord}

\def\Hom{{\rm Hom}}
\def\End{{\rm End}}

\def\Spec{{\rm Spec}}

\def\char{{\rm char \hspace{1mm}}}

\numberwithin{equation}{section}

\newtheorem{theorem}{Theorem}[section]
\newtheorem{thm}{Theorem}[section]
\newtheorem{example}[theorem]{Example}
\newtheorem{lemma}[theorem]{Lemma}

\newtheorem{remark}[theorem]{Remark}

\newtheorem{corollary}[theorem]{Corollary}
\newtheorem{proposition}[theorem]{Proposition}

\usepackage{tikz}

\DeclareMathOperator{\Jac}{Jac}

\newcommand{\calO}{\mathcal{O}}
\DeclareMathOperator{\HH}{H}

\begin{document}
\setlength{\parskip}{2pt} 
\setlength{\parindent}{8pt}
\title{Vanishing criteria for Ceresa cycles} 
\author{Jef Laga}
\address{Department of Pure Mathematics and Mathematical Statistics, Wilberforce Road, Cambridge, CB3 0WB, UK}
\email{jeflaga@hotmail.com}
\author{Ari Shnidman}
\address{Einstein Institute of Mathematics, Hebrew University of Jerusalem, Israel} 
\email{ari.shnidman@gmail.com}
\maketitle

\begin{abstract}
We prove cohomological vanishing criteria for the Ceresa cycle of a curve $C$ embedded in its Jacobian $J$. Namely:  
\begin{enumerate}[$(A)$]
\item If $\HH^3(J)^{\Aut(C)} = 0$, then the Ceresa cycle is torsion modulo rational equivalence; 
\item If $\HH^0(J, \Omega_J^3)^{\Aut(C)} = 0$, then the Ceresa cycle is torsion modulo algebraic equivalence,
\end{enumerate}
with criterion $(B)$ conditional on the Hodge conjecture.  

We then use these criteria to study the simplest family of curves where $(B)$ holds but $(A)$ does not, namely the family of Picard curves $C \colon y^3 = x^4 + ax^2 + bx + c$. Criterion $(B)$ and work of Schoen combine to show that the Ceresa cycle of a Picard curve is torsion in the Griffiths group. We furthermore determine exactly when it is torsion in the Chow group. As a byproduct, we deduce: there exist one-parameter families of plane quartic curves with torsion Ceresa Chow class; the torsion locus in $\mathcal{M}_3$ of the Ceresa Chow class contains infinitely many components; the order of a torsion Ceresa Chow class of a Picard curve over a number field $K$ is bounded, with the bound depending only on $[K\colon \mathbb{Q}]$.
Finally, we determine which automorphism group strata are contained in the vanishing locus of the universal Ceresa cycle over $\mathcal{M}_3$. 
\end{abstract}

\makeatletter
\makeatother

\section{Introduction}

Let $k$ be an algebraically closed field and $C$ a smooth, projective, and connected curve over $k$ of genus $g \geq 2$ with Jacobian variety $J$.
Let $e$ be a degree-$1$ divisor of $C$ and let $\iota_e \colon C \hookrightarrow J$ be the Abel-Jacobi map based at $e$. 
We study the torsion behaviour of the Ceresa cycle 
\begin{align}\label{equation: ceresa cycle}
    \kappa_{C,e} := [\iota_e(C)]- (-1)^*[\iota_e(C)] \in \CH_1(J)
\end{align}
in the Chow group modulo rational equivalence.
If $\kappa_{C,e}$ is torsion, then $(2g-2)e= K_C$ in $\CH_0(C)\otimes \Q$, where $K_C$ is the canonical divisor class.
Moreover if $(2g-2)e$ is canonical, then the image of $\kappa_{C,e}$ in $\CH_1(J)\otimes{\Q}$ is independent of $e$ and we denote it by $\kappa(C)$, see \S\ref{subsec: ceresa cycles} for these claims.
Thus, the class $\kappa(C)$ vanishes if and only if $\kappa_{C,e}$ is torsion for some degree-$1$ divisor $e$.

We also consider the image $\bar{\kappa}(C)$ of $\kappa(C)$ in the Griffiths group $\Gr_1(J)\otimes \Q$ of homologically trivial $1$-cycles modulo algebraic equivalence. 
When $g = 2$, or more generally when $C$ is hyperelliptic, it is easy to see that $\kappa(C) = 0$.  
On the other hand, Ceresa famously showed that $\bar{\kappa}(C) \neq 0$ for a very general curve $C$ over $\C$ of genus $g \geq 3$ \cite{Ceresa}.

The vanishing of the Ceresa cycle is interesting for various reasons. For example, $\kappa(C) = 0$ if and only if the Chow motive $\frak{h}(C)$ has a multiplicative Chow-K\"unneth decomposition (by \cite[Proposition 3.1]{FuLaterveerVial} and Proposition \ref{proposition: vanishing ceresa equivalent to vanishing C1}).
Moreover $\bar{\kappa}(C)=0$ if and only if the tautological subring modulo algebraic equivalence is generated by a theta divisor (by \cite[Corollary 3.4]{Beauville-algebraiccyclesonjacobianvarieties}), in which case Poincar\'e's formula $[C] = \frac{\Theta^{g-1}}{(g-1)!}$ holds modulo algebraic equivalence. 
More generally, the Ceresa cycle over $\mathcal{M}_g$ serves as a testing ground for the study of homologically trivial algebraic cycles in codimension greater than 1.    

\subsection{Vanishing criteria} We prove cohomological vanishing criteria for Ceresa cycles of curves with nontrivial automorphisms. Let $\HH^*(-)$ be a Weil cohomology functor, such as $\ell$-adic cohomology with $\ell \neq \char(k)$ or singular cohomology when $k = \C$.  Note that the finite group $\Aut(C)$ acts on $\HH^*(J)$, by functoriality.  Cupping with the principal polarization gives an injection $\HH^1(J)(-1) \hookrightarrow \HH^3(J)$, allowing us to define the primitive cohomology $\HH^3(J)_{\mathrm{prim}} := \HH^3(J)/\HH^1(J)(-1)$. 

\begin{thm}\label{thm: main}
If $\HH^3(J)_{\mathrm{prim}}^{\Aut(C)} = 0$, then $\kappa(C) = 0$.   \end{thm}

This improves on a recent result of Qiu and Zhang stating that if $(\HH^1(C)^{\otimes 3})^{\Aut(C)} = 0$, then $\kappa(C) = 0$ \cite{QiuZhang}.\footnote{
This formulation is equivalent to theirs, using \cite[Theorem 1.5.5]{ShouwuZhang}.
} 
By contrast, Theorem \ref{thm: main} requires only the weaker condition that the subrepresentation $\HH^3(J)_{\mathrm{prim}} \subset \HH^3(J) \simeq \bigwedge^3 \HH^1(C) \subset \HH^1(C)^{\otimes 3}$ has no nontrivial $\Aut(C)$-fixed points.  If the quotient $C/\Aut(C)$ has genus $0$, then our hypothesis is equivalent to $\HH^3(J)^{\Aut(C)} = 0$.  

Our proof of Theorem \ref{thm: main} is inspired by Beauville's proof that for the curve $y^3 = x^4 + x$, the image of $\kappa(C)$ under the complex Abel-Jacobi map vanishes \cite{Beauville}. To achieve a vanishing result in the Chow group, we work directly with the rational Chow motive $\frak{h}^3(J)$ and make crucial use of the finite-dimensionality results of Kimura \cite{Kimura}.

Our second result is a vanishing criterion for $\overline{\kappa}(C)$, however the result  is conditional on the Hodge conjecture. In particular, we assume for the rest of the introduction that $k$ has characteristic $0$. 

\begin{thm}\label{thm: griffiths vanishing}
    Assume the Hodge conjecture for abelian varieties. If $\HH^0(J, \Omega_J^3)^{\Aut(C)} = 0$, then $\bar{\kappa}(C) = 0$. 
\end{thm}
More precisely, we require the Hodge conjecture for $J\times A$, where $A$ is the abelian variety described in Proposition \ref{prop: preparation griffiths vanishing}.
Note that $\HH^0(J, \Omega_J^3) \simeq \bigwedge^3\HH^0(C, \Omega_C^1)$, so the conditions of Theorems \ref{thm: main} and \ref{thm: griffiths vanishing} both depend only on the abstract representation $(G,V) = (\Aut(C), \HH^0(C, \Omega_C^1))$. 

The proof of Theorem \ref{thm: griffiths vanishing} is in the same spirit as that of Theorem \ref{thm: main}. The Hodge conjecture is used to show that the motive $\frak{h}^3(J)^{\Aut(C)}$ is isomorphic to $\frak{h}^1(A)(-1)$, from which the algebraic triviality of $\kappa(C)$ follows. 

\subsection{Picard curves}

The condition $\HH^3(J)_{\mathrm{prim}}^{\Aut(C)} = 0$ in Theorem \ref{thm: main} is only rarely satisfied, e.g.\ it holds for exactly two plane quartic curves over $\C$ (see Theorem \ref{theorem: main theorem E} below).  
The condition $\HH^0(J, \Omega_J^3)^{\Aut(C)}=0$ of Theorem \ref{thm: griffiths vanishing} is satisfied more often; e.g.\ it holds for all Picard curves $y^3 = x^4 + ax^2 + bx + c$.
For such curves, it turns out that we require a nontrivial case of the Hodge conjecture to apply Theorem \ref{thm: griffiths vanishing}, namely that a certain Weil class on an abelian fourfold is algebraic.
Fortunately, Schoen has proved the Hodge conjecture in our specific situation \cite{schoen-addendumhodgeclassesselfproducts}. (See also recent work of Markman \cite{Markman-monodromygeneralized}.) 
This leads to an unconditional proof of the vanishing of $\bar{\kappa}(C)$.  By further analyzing $\kappa(C)$, we prove the following precise characterization of the vanishing of $\kappa(C)$ in the Picard family:
\begin{thm}\label{thm: Picard main}
    Let $C_f$ be a smooth projective Picard curve with model 
    \begin{align}\label{equation: picard curve}
    C_f \colon y^3 = f(x) = x^4+ ax^2+bx+c
    \end{align}
    for some $a,b,c\in k$. Consider the point $P_f= (a^2+12c,72ac-2a^3-27b^2)$ on the elliptic curve 
    \[E_f \colon y^2 = 4x^3-27\cdot \disc(f).\] 
    Then $\bar{\kappa}(C_f)=0$, and $\kappa(C_f)=0$ if and only if $P_f \in E_f(k)$ is torsion.
\end{thm}

Here, $\disc(f)$ is the discriminant of $f$ and the coordinates of $P_f$ are the usual $I$- and $J$-invariants of the binary quartic form $z^4 f(x/z)$; see \S\ref{subsec: generalities}.
The fact that $P_f$ defines a point on $E_f$ follows from a classical relation between the invariants of a binary quartic. Picard curves have a unique point $\infty$ at infinity, for which $(2g-2)\infty = 4\infty$ is canonical. For a more precise version of Theorem \ref{thm: Picard main} relating the torsion orders of $\kappa_{C_f,\infty}$ and $P_f$, see \S\ref{subsec: ceresa vanishing criteria}.

Theorem \ref{thm: Picard main} generalizes the results of \cite{LagaShnidman-CeresaBiellipticPicard} concerning \emph{bielliptic} Picard curves, i.e.\ those with $b=0$ in \eqref{equation: picard curve}.
There we exploited the explicit geometry of the bielliptic cover to characterize the vanishing locus of $\kappa(C)$.
Here we must proceed more indirectly: using Schoen's work on the Hodge conjecture, we first show that the vanishing of $\kappa(C)$ is equivalent to the vanishing of its Abel--Jacobi image in the intermediate Jacobian. 
Using the Shioda-Tate formula and considerations of isotrivial families of elliptic curves, we are able to bootstrap the calculations of \cite{LagaShnidman-CeresaBiellipticPicard} to  determine the Abel-Jacobi image up to multiple and obtain the characterization of Theorem \ref{thm: Picard main}.

In \S\ref{sec: proofs theorems A and B}, we give several examples of higher genus families of curves $\mathcal{C} \rightarrow B$ for which Theorem \ref{thm: griffiths vanishing} applies. It would be interesting to completely characterize the vanishing of $\kappa(C)$ in these families as well, and to find more such families.  
The general proof strategy would follow that of Theorem \ref{thm: Picard main}: using the Hodge conjecture, one shows that the vanishing of $\kappa(\mathcal{C}_b)$ is equivalent to the vanishing of $\sigma(b)$, where $\sigma$ is a section of a certain abelian scheme $\mathcal{A} \rightarrow B$.
(The analytification of $\sigma$ is the normal function associated to the Ceresa cycle over $B$, which lands in the $\Aut(C)$-fixed points of the intermediate Jacobian.) One then tries to identify $\mathcal{A} \rightarrow B$  and $\sigma$ explicitly. 

\subsection{Vanishing loci in genus $3$}\label{subsec: intro vanishing loci in genus 3}
Let $\mathcal{M}_g$ be the coarse moduli space of genus-$g$ curves $C$, seen as a variety over $\Q$.
Let $V_g^{\mathrm{rat}}\subset \mathcal{M}_g$ be the locus of curves with $\kappa(C) = 0$.
Let $V_g^{\mathrm{alg}}\subset \mathcal{M}_g$ be the analogous subset for $\bar{\kappa}(C)$. 
The subsets $V_g^{\mathrm{rat}}$ and $V_g^{\mathrm{alg}}$ are each a countable union of closed algebraic subvarieties (Lemma \ref{lemma: Zg,barZg countable union}). We have $V_g^{\mathrm{rat}}\subset V_g^{\mathrm{alg}}$, and Ceresa's result \cite{Ceresa} shows that $V_g^{\mathrm{alg}} \neq \mathcal{M}_g$ for all $g\geq 3$.
What can be said about the irreducible components of $V_g^{\mathrm{rat}}$ and $V_g^{\mathrm{alg}}$ and their dimensions? For a fixed $g$, are there only finitely many components?  
Collino and Pirola showed that $V_3^{\mathrm{rat}}$ does not contain subvarieties of dimension $\geq 4$ that are not themselves contained in the hyperelliptic locus \cite[Corollary 4.3.4]{ColinoPirola-infinitesimal}.
On the other hand, Theorem \ref{thm: Picard main} shows that $V_3^{\mathrm{alg}}$ contains the $2$-dimensional Picard locus. 
Turning to $V_3^{\mathrm{rat}}$, we answer several open question about its geometry and arithmetic by analyzing the torsion locus of the section $P_f$ of $E_f$:



\begin{thm}\label{thm: corollaries}
Let $X\subset \mathcal{M}_3$ be the Picard locus.
    \begin{enumerate}[$(i)$]
    \item $(X\cap V_3^{\mathrm{rat}})\times_{\Q} \mathbb{C}$ is a countably infinite disjoint union of rational curves. In particular, there exist one-parameter families of plane quartic curves in $V_3^{\mathrm{rat}}$,  
    and $V_3^{\mathrm{rat}}$ contains infinitely many positive-dimensional components.
    \item If $K$ is a number field and $C \in (X \cap V_3^{\mathrm{rat}})(K)$, then the order of $\kappa_{C,\infty}$ in $\CH_1(J)$ is bounded above by a quantity depending only on $[K\colon \Q]$.
\end{enumerate}
\end{thm}

An example of a one-parameter family as in $(i)$ is the family $C_t \colon y^3 = x^4 -12x^2 + tx -12$. 
Note that we have previously shown in \cite[Theorem 1.3]{LagaShnidman-CeresaBiellipticPicard} that the torsion orders of the classes $\kappa_{C,e}$, where $C$ has genus $3$ and $4e= K_C$, are unbounded.

Recall that there is a stratification of the non-hyperelliptic locus of $\mathcal{M}_3$ by locally closed subvarieties $X_G$, indexed by certain finite groups $G$, with the property that a non-hyperelliptic curve $[C]\in \mathcal{M}_3(\C)$ lies in $X_G$ if and only if $\Aut(C)\simeq G$. It turns out that the isomorphism class of the $\Aut(C)$-representation $\HH^0(C, \Omega_C^1)$ does not depend on the choice of curve $[C]$ in $X_G(\C)$.
Our final theorem determines exactly which $X_G$ are contained in $V_3^{\mathrm{rat}}$ or $V_3^{\mathrm{alg}}$.

\begin{thm}\label{theorem: main theorem E}
    Let $G$ be a finite group isomorphic to the automorphism group of a plane quartic. 
    Then $X_G\subset V_3^{\mathrm{rat}}$ if and only if $G=C_9$ or $G_{48}$, and $X_G\subset V_3^{\mathrm{alg}}$ if and only if $X_G$ is contained in the Picard locus, in other words if and only if $G= C_3,C_6,C_9$ or $G_{48}$.
\end{thm}

Here, $G_{48}$ is the group with GAP label $(48,33)$. The strata $X_{C_{9}}$ and $X_{G_{48}}$ are single closed points, represented by the curves $y^3= x^4+x$ and $y^3 = x^4+1$ respectively. 

Theorem \ref{theorem: main theorem E} shows that $X_G\subset V_3^{\mathrm{rat}}$ if and only if $\HH^3(J)^G=0$ for every (equivalently, some) curve in $X_G$, and that $X_G\subset V_3^{\mathrm{alg}}$ if and only if $\HH^0(J, \Omega_J^3)^G=0$ for every curve in $X_G$. Thus, for $g = 3$, the criteria of Theorems \ref{thm: main} and \ref{thm: griffiths vanishing} exactly single out the strata $X_G$ where $\kappa(C)$ or $\bar{\kappa}(C)$ identically vanishes.
Does this continue to hold in higher genus? Are there non-hyperelliptic curves of arbitrarily large genus for which Theorem \ref{thm: main} or \ref{thm: griffiths vanishing} applies?

We also use Theorem \ref{thm: griffiths vanishing} to find the first examples of positive dimensional components of $V_g^{\mathrm{alg}}$ with generic point satisfying $\Aut(C) = \{1\}$.\footnote{Qiu and Zhang have previously found individual curves with $\kappa(C) =0 $ and $\Aut(C) = \{1\}$ \cite{QiuZhangII}.} This uses the following fact: if $C$ dominates another curve $D$, then $\overline{\kappa}(C)=0$ implies $\overline{\kappa}(D) = 0$. For example, the genus 6 family 
  \[C_t \colon y^9 = \left(\frac{x+1}{x-1}\right) \left(\frac{x+t}{x-t}\right)^3,\]    
  admits a $D_9$-action and lies in $V_6^{\mathrm{alg}}$ (conditional on the Hodge conjecture). The quotient by the involution $(x,y) \mapsto (-x,y^{-1})$ is the genus 3 family
  \[X_u \colon y^4 = xy^3 + u x^3 + 3y^2 - 2u x - u y - 1,\]
where $u = (t-1)/(t+1)$, which lies in $V_3^{\mathrm{alg}}$ by the fact above, and satisfies $\Aut(X_u) = \{1\}$ generically. A similar construction gives a 1-dimensional component of $V_4^{\mathrm{alg}}$ (again conditional on the Hodge conjecture) with trivial generic automorphism group; see Theorem \ref{thm: trivial automorphism group example}.  Thus, even in genus 3, the implications of Theorems \ref{thm: main} and \ref{thm: griffiths vanishing} are not yet fully understood, since it is hard to know which curves admit high genus covers that satisfy our vanishing criteria. We should remark that, as far as we know, there are no known examples of curves $C$ with $\overline{\kappa}(C) = 0$, $\Aut(C) = \{1\}$, and $\End(J) = \Z$; see \cite{EllenbergLoganSrinivasan} for data confirming this for many quartic plane curves. The curves $X_u$ are no exceptions as generically we have $\End^0(\Jac(X_u)) = \Q(\zeta_9 + \zeta_9^{-1})$ by Remark \ref{rem: costa}.  

Finally, Gao and Zhang have recently proven a Northcott property for the Beilinson-Bloch height of the Ceresa cycle on $J$ and (equivalently) the modified diagonal cycle $\Delta(C)$ on $C^3$ \cite{GaoZhang-NorthcottCeresa}. More precisely, for each $g \geq 3$, there exists an open dense subset $U_g \subset \mathcal{M}_g$ such that for any $X \in \R$ and $d \in \N$, the number of $C \in U_g(\bar\Q)$ defined over a number field of degree at most $d$ and with $\langle \Delta(C), \Delta(C)\rangle < X$ is finite. In order to better understand Ceresa vanishing loci in families of curves, it is of great interest to try to identify the largest such open dense set $U_g$, or equivalently, its complement $Z_g^{\mathrm{slim}} := \mathcal{M}_g \setminus U_g$.  The Northcott property implies that any positive dimensional component of $V_g^{\mathrm{rat}}$ is contained in $Z_g^{\mathrm{slim}}$, but $Z_g^{\mathrm{slim}}$ may be strictly larger than the union of the positive dimensional components of $V_g^{\mathrm{rat}}$. Indeed, Theorem \ref{thm: corollaries} shows that $X_{C_3} \subset Z_3^{\mathrm{slim}}$, even though $X_{C_3} \not\subset V_3^{\mathrm{rat}}$.

\subsection{Structure of paper}
In \S\ref{sec: notation and background} we collect some standard results on Chow groups, Chow motives and Ceresa cycles. 
In \S\ref{sec: proofs theorems A and B} we prove the cohomological vanishing criteria (Theorems \ref{thm: main} and \ref{thm: griffiths vanishing}) and give some examples.
In \S\ref{sec: Picard} we study the family of Picard curves in detail and prove Theorems \ref{thm: Picard main} and \ref{thm: corollaries}.
Finally, in \S\ref{sec: automorphism strata and Ceresa vanishing genus 3} we introduce the Ceresa vanishing loci $V_g^{\mathrm{rat}},V_g^{\mathrm{alg}}\subset \mathcal{M}_g$, prove Theorem \ref{thm: corollaries} and determine which automorphism strata they contain in genus $3$, proving Theorem \ref{theorem: main theorem E}.

\subsection{Acknowledgements}
We thank Jeffrey Achter, Irene Bouw, Edgar Costa, Ben Moonen, Burt Totaro and Congling Qiu for helpful conversations.
This research was carried out while the first author was a Research Fellow at St John's College, University of Cambridge, which he thanks for providing excellent working conditions.
The second author was funded by the European Research Council (ERC, CurveArithmetic, 101078157).

\section{Notation and background}\label{sec: notation and background}

\subsection{Chow groups}
Let $k$ be field. A variety is by definition a separated scheme of finite type over $k$. We say a variety is nice if it is smooth, projective and geometrically integral. 
If $X$ is a smooth and geometrically integral variety and $p\in \{0,\dots,\dim X\}$, let $\CH^p(X)$ denote the Chow group (with $\Z$-coefficients) of codimension $p$ cycles modulo rational equivalence.
If $Z\subset X$ is a closed subscheme of codimension $p$, we denote its class in $\CH^p(X)$ by $[Z]$ (using \cite[\href{https://stacks.math.columbia.edu/tag/02QS}{Tag 02QS}]{stacks-project} if $Z$ is not integral).

If $X$ is additionally projective, then $\CH^p(X)$ has a filtration by subgroups
\begin{align*}
    \CH^p(X)_{\mathrm{alg}} \subset \CH^p(X)_{\mathrm{hom}} \subset \CH^p(X),
\end{align*}
where $\CH^p(X)_{\mathrm{alg}}$ is the subgroup of algebraically trivial cycles (in the sense of \cite[\S3.1]{ACMV-parameterspacesalgebraic}) and $\CH^p(X)_{\mathrm{hom}}$ the subgroup of homologically trivial cycles (with respect to a fixed Weil cohomology theory for nice varieties over $k$).
The Griffiths group is by definition $\Gr^p(X) = \CH^p(X)_{\mathrm{hom}}/\CH^p(X)_{\mathrm{alg}}$.
We occasionally write $\CH_p(X) = \CH^{\dim X -p}(X)$ and $\Gr_p(X) = \Gr^{\dim X -p}(X)$.
If $R$ is a ring, we write $\CH^p(X)_R= \CH^p(X)\otimes_{\Z} R$ and $\Gr^p(X)_R = \Gr^p(X)\otimes_{\Z} R$.

\subsection{Base change, specialization and families}

We state three lemmas concerning operations on cycles.
These seem to be standard, but we could not locate proofs in the literature.


\begin{lemma}\label{lemma: base extension chow groups}
    Let $X/k$ be a nice variety and $K/k$ a (not necessarily finite) extension of fields.
    \begin{enumerate}
        \item The base change maps $\CH^p(X)_\Q \rightarrow \CH^p(X_K)_\Q$ and $\Gr^p(X)_\Q\rightarrow \Gr^p(X_K)_\Q$ are injective.
        \item If in addition $k$ is algebraically closed, the base change maps $\CH^p(X) \rightarrow \CH^p(X_K)$ and $\Gr^p(X)\rightarrow \Gr^p(X_K)$ are injective.
    \end{enumerate}
\end{lemma}
\begin{proof}
    The proof is an adaptation of \cite[Lemma 1A.3, p.\ 22]{bloch-lecturesonalgebraiccycles}.
    We first prove $(2)$ for $\CH^p(X)\rightarrow \CH^p(X_K)$.
    Suppose $\alpha\in \CH^p(X)$ has trivial image in $\CH^p(X_K)$.
    Then $\alpha$ already has trivial image in $\CH^p(X_{K'})$, where $K'\subset K$ is a subfield that is finitely generated over $k$, since the data witnessing triviality in $\CH^p(X_K)$ can be defined over such a subfield.
    By spreading out, we can find a smooth integral variety $U/k$ with function field $K'$ such that $\alpha$ has trivial image in $\CH^p(X\times_k U)$.
    Since $k$ is algebraically closed, there exists a $k$-point $u\in U(k)$. 
    Pulling back along $u$ defines a left-inverse $\CH^p(X\times_k U) \rightarrow \CH^p(X)$ to the map $\CH^p(X) \rightarrow \CH^p(X\times_k U)$.
    It follows that $\alpha$ is trivial in $\CH^p(X)$, as desired. 
    The argument for $\Gr^p(X)$ is identical and omitted.

    We now prove $(1)$ for $\CH^p(X)_{\Q}\rightarrow \CH^p(X_K)_{\Q}$. There exists a field $L$ containing both $K$ and an algebraic closure $\bar{k}$ of $k$. It therefore suffices to prove the two base change maps $\CH^p(X)_\Q \rightarrow \CH^p(X_{\bar{k}})_\Q \rightarrow \CH^p(X_L)_{\Q}$ are both injective. The first one follows from the fact that for a finite extension $k'/k$ the pushforward map $\CH^p(X_{k'}) \rightarrow \CH^p(X)$, when precomposed with the base change map, is multiplication by $[k':k]$. 
    The second follows from Part $(2)$.
    The case of $\Gr^p(X)_{\Q}$ is again analogous.
\end{proof}

To discuss specialization in the next lemma, let $R$ be a discrete valuation ring with fraction field $K$ and residue field $k$.
Let $X\rightarrow \Spec (R)$ be a smooth, projective morphism with geometrically integral fibers, so the generic and special fibers $X_K$ and $X_k$ are nice varieties over $K$ and $k$ respectively. 
In this setting, Fulton has defined \cite[\S20.3]{Fulton-intersectiontheory} a specialization morphism $\mathrm{sp}\colon \CH^p(X_K) \rightarrow \CH^p(X_k)$ for every $0\leq p\leq \dim(X_K)$.
It has the property that if $Z\subset X$ is a closed integral subscheme of codimension $p$, flat over $R$, then $\mathrm{sp}([Z_K]) = [Z_k]$.

\begin{lemma}\label{lemma: specialization morphism}
    In the above notation, $\mathrm{sp}$ sends $\CH^p(X_K)_{\mathrm{alg}}\otimes \Q$ to $\CH^p(X_k)_{\mathrm{alg}}\otimes \Q$.
\end{lemma}
\begin{proof}
    Let $C/K$ be a nice curve with $K$-points $t_0, t_1\in C(K)$ and let $Z\subset (X_K)\times_K C$ be an integral closed subscheme of codimension $p$, flat over $C$.
    By \cite[Theorem 1]{ACMV-parameterspacesalgebraic} and the remarks thereafter, it suffices to prove that $\mathrm{sp}([Z_{t_0}] - [Z_{t_1}]) \in \CH^p(X_k)_{\mathrm{alg}} \otimes \Q$ for each such tuple $(C, t_0,t_1, Z)$.
    
    Let $\mathcal{C}\rightarrow \Spec(R)$ be an integral, projective, flat, proper and regular model of $C$, which exists by \cite[Proposition 10.1.8]{liu-algebraicgeometryarithmeticcurves}.
    Let $\mathrm{Hilb}_{X/R}\rightarrow \Spec(R)$ be the Hilbert scheme of the projective morphism $X\rightarrow \Spec(R)$. 
    The closed subscheme $Z$ determines a $K$-morphism $\alpha\colon C\rightarrow \mathrm{Hilb}_{X/R}$.
    Since $C$ is connected, the image of $\alpha$ is contained in an open and closed subscheme $\mathrm{Hilb}_{X/R}^{\phi}$ with fixed Hilbert polynomial.
    The surface $\mathcal{C}$ is regular, $\mathrm{Hilb}^{\phi}_{X/R}\rightarrow \Spec(R)$ is projective \cite[Corollary (2.8)]{AltmanKleiman-compactifyingpicardscheme} and we may view $\alpha$ as a rational map $\mathcal{C}\dashrightarrow \mathrm{Hilb}_{X/R}^{\phi}$.
    Therefore by \cite[Theorem 9.2.7]{liu-algebraicgeometryarithmeticcurves} there exists a birational map $\widetilde{\mathcal{C}} \rightarrow \mathcal{C}$, obtained by successively blowing up closed points in the special fiber, and an $R$-morphism $\tilde{\alpha}\colon \widetilde{\mathcal{C}} \rightarrow \mathrm{Hilb}^{\phi}_{X/R}$ with generic fiber $\alpha$.
    In other words, there exists a closed subscheme $\mathcal{Z} \subset \widetilde{\mathcal{C}} \times_R X$, flat over $\widetilde{\mathcal{C}}$, whose generic fiber equals $Z$.
    
    The $K$-points $t_i\in C(K)$ extend to $R$-points $\tilde{t}_i\colon \Spec(R) \rightarrow \widetilde{\mathcal{C}}$ whose reductions $\bar{t}_i\in \widetilde{\mathcal{C}}_k(k)$ land in the smooth locus of $\widetilde{\mathcal{C}}_k$.
    The closed subschemes $\mathcal{Z}_{\tilde{t}_i}\subset X$ are flat over $R$.
    Therefore $\mathrm{sp}([Z_{t_0}]-[Z_{t_1}]) = [\mathcal{Z}_{\bar{t}_0}] - [\mathcal{Z}_{\bar{t}_1}]$.
    
    By the Zariski connectedness theorem, the special fiber $\widetilde{\mathcal{C}}_k$ is connected. 
    Consequently, by resolving irreducible components of $\widetilde{\mathcal{C}}_k$, we can connect $\bar{t}_0$ to $\bar{t}_1$ by a sequence of nice curves: there exists a finite extension $k'/k$, a collection of nice curves $D_1, \dots,D_n$ over $k'$, and for each $1\leq i \leq n$ a pair of points $s_{i,1}, s_{i,2}\in D_i(k')$ and a morphism $\varphi_i\colon D_i \rightarrow \mathcal{C}_{k'}$ such that $\varphi_1(s_{1,1}) = \bar{t}_1$, $\varphi_i(s_{i,2}) = \varphi_{i+1}(s_{i+1,1})$ for all $1\leq i\leq n-1$ and $\varphi_{n}(s_{n,2}) = \bar{t}_2$.
    Letting $W^{(i)}$ be the pullback of $\mathcal{Z}_k$ along $\varphi_i$, 
    we have
    \begin{align*}
        \mathrm{sp}([Z_{t_0}]-[Z_{t_1}])_{k'} = ([W^{(1)}_{s_{1,1}}]- [W^{(1)}_{s_{1,2}}]) + 
        \cdots + ([W^{(n)}_{s_{n,1}}]- [W^{(n)}_{s_{n,2}}]).
    \end{align*}
    Since each term $([W^{(i)}_{s_{i,1}}]- [W^{(i)}_{s_{i,2}}])$ lies in $\CH^p(X_{k'})_{\mathrm{alg}}\otimes \Q$ by definition of algebraic triviality, the same is true for their sum. 
    Taking the pushforward along $X_{k'} \rightarrow X_k$, we see that $[k':k] \cdot \mathrm{sp}([Z_{t_0}]-[Z_{t_1}]) \in \CH^p(X_k)_{\mathrm{alg}}$, as desired.
\end{proof}

\begin{remark}
    {\em
    The proof shows that if $K$ is perfect and $k$ is algebraically closed, then $\mathrm{sp}$ even sends $\CH^p(X_K)_{\mathrm{alg}}$ to $\CH^p(X_k)_{\mathrm{alg}}$.
    }
\end{remark}

\begin{lemma}\label{lemma: locus of rational triviality countable union}
    Let $X\rightarrow S$ be a smooth proper morphism of smooth varieties over a field $k$.
    Let $\alpha$ be a codimension $p$ cycle on $X$.
    Then the locus of points $s\in S$ such that the (Gysin) fiber $\alpha_s\in \CH^p(X_s)_{\Q}$ is zero (resp.\ lies in $\CH^p(X_s)_{\mathrm{alg}, \Q}$) is a countable union of closed algebraic subvarieties of $S$.
\end{lemma}
\begin{proof}
    There exists a countable subfield $k_0\subset k$, a smooth proper morphism $X_0\rightarrow S_0$ of varieties over $k_0$ and a cycle $\alpha_0$ on $X_0$ whose base change to $k$ are $X\rightarrow S$ and $\alpha$ respectively. 
    Since $k_0$ is countable, $S_0$ has only countably many closed subschemes. 
    Let $\mathcal{F}$ be the collection of integral closed subschemes $Z\subset S_0$ such that $\alpha_0$ is zero in $\CH^p(X_{0,\eta(Z)})_{\Q}$, where $\eta(Z)\in S_0$ denotes the generic point of $Z$.
    Then we claim that the locus of $s\in S$ for which $\alpha_s\in \CH^p(X_s)_{\Q}$ is zero is exactly the union $\bigcup_{Z\in \mathcal{F}} Z_k \subset S$.
    This follows from Lemmas \ref{lemma: base extension chow groups} and properties of the specialization map; we omit the details, and the similar argument for algebraic triviality.
\end{proof}

\subsection{Chow motives}\label{subsec: chow motives}
We recall a few relevant facts about the category  $\mathsf{Mot}(k)$ of (pure, contravariant) Chow motives $M$ over $k$ with $\Q$-coefficients; see \cite{Scholl-classicalmotives} or \cite[\S2]{MurreNagelPeters} for basic definitions. 

Denote the Lefschetz motive by $\mathbb{L}$, its $n$th tensor power by $\mathbb{L}^n$ and write $M(n) = M\otimes \mathbb{L}^n$.
Following \cite[\S2.5]{MurreNagelPeters}, the Chow group in codimension $p$ of $M$ is by definition $\CH^p(M) = \Hom_{\mathsf{Mot}(k)}(\mathbb{L}^{p}, M)$.
If $M = \frak{h}(X)$ where $X$ is smooth projective over $k$, then $\CH^p(M) = \CH^p(X)_{\Q}$. (Beware that we need to take $\Q$-coefficients on the right hand side.)
If $C$ is a nice curve over $k$ and $\varphi\colon \frak{h}(C)\rightarrow M(p-1)$ a morphism in $\mathsf{Mot}(k)$, we obtain a homomorphism of abelian groups $\CH^1(\varphi)\colon \CH^1(C)_{\Q} \rightarrow \CH^p(M)$.
We define $\CH^p(M)_{\mathrm{alg}}$ to be the union of the images of $\CH^1(\varphi)$ ranging over all pairs $(C, \varphi)$ as above.
This coincides with $\CH^p(X)_{\mathrm{alg}} \otimes \Q$ when $M = \frak{h}(X)$, see \cite[Corollary 3.13]{ACMV-parameterspacesalgebraic}.

Finally, we define motives of fixed points. 
If $G$ is a finite group acting on a motive $M$, write $M^G$ for the submotive cut out by the idempotent $\frac{1}{\#G}\sum_{g \in G} g_* \in \End(M)$. 
We have $\CH^p(M^G) = \CH^p(M)^G$.
If $G$ acts on a nice variety $X$, then $G$ acts on $\frak{h}(X)$ and $\CH^p(X)$, and $\CH^p(\frak{h}(X)^G) = \CH^p(X)^G\otimes {\Q}$.

\subsection{Chow--K\"unneth decomposition for abelian varieties}\label{subsec: chow kunneth decomposition abelian vars}

Let $A/k$ be a $g$-dimensional abelian variety.
Deninger and Murre \cite[\S3]{DeningerMurre} have constructed a canonical Chow--K\"unneth decomposition 
\begin{align}\label{equation:deninger-murre decomposition}
 \frak{h}(A) = \bigoplus_{i = 0}^{2g}\frak{h}^i(A),   
\end{align}
uniquely characterized by the following property: if $(n)\colon A\rightarrow A$ denotes the multiplication-by-$n$, then $(n)^*$ acts on $\frak{h}^i(A)$ via $n^i$ for every integer $n$. 
On the other hand, Beauville \cite{Beauville-surlanneadechow} has shown that there exists a direct sum decomposition $\CH^p(A)_\Q = \bigoplus_{s = p-g}^p \CH^p_{(s)}(A)$, where 
\begin{align}\label{equation: beauville decomposition}
\CH^p_{(s)}(A) = \{\alpha \in \CH^p(A)_\Q \colon (n)^*\alpha = n^{2p-s} \alpha \quad \forall n\in \Z\}.
\end{align}
The two decompositions are linked by the formula $\CH^p(\frak{h}^i(A)) = \CH^p_{(2p-i)}(A)$.
Beauville conjectured that $\CH^p_{(s)}(A) = 0$ when $s<0$, and he proved it when $p \in \{0,1,g-2,g-1,g\}$ \cite[Proposition 3(a)]{Beauville-surlanneadechow}.

\begin{example}
    {\em 
    For $p = 1$, the Beauville decomposition $\CH^1(A)_\Q = \CH^1_{(0)}(A) \oplus \CH^1_{(1)}(A)$ is the decomposition of a divisor class into symmetric and anti-symmetric classes. 
    }
\end{example}

\begin{lemma}\label{lemma: Griffiths group vanishes for h^1}
    If $A/k$ is an abelian variety, then $\CH^p(\frak{h}^1(A))_{\mathrm{alg}} = \CH^p(\frak{h}^1(A))$ for all $p\geq 0$.
\end{lemma}
\begin{proof}
    By Lemma \ref{lemma: base extension chow groups}(1), we may assume $k$ is algebraically closed.
    The only nonzero Chow group of $\frak{h}^1(A)$ is $\CH^1(\frak{h}^1(A)) = \CH^1_{(1)}(A)$, the set of anti-symmetric elements of $\CH^1(A)_{\Q}$.
    The lemma follows from the fact that $\CH^1_{(1)}(A) = \CH^1(A)_{\mathrm{hom}}\otimes \Q$ and that homological and algebraic equivalence coincide for codimension-$1$ cycles.
\end{proof}

\subsection{The Lefschetz decomposition for abelian varieties}\label{subsec: lefschetz decomposition abelian varieties}

Let $A/k$ be an abelian variety with polarization $\lambda\colon A \rightarrow A^{\vee}$.
Let $\ell \in \CH^1_{(0)}(A)$ be the unique class satisfying $2\ell = (1, \lambda)^* \mathcal{P}$, where $\mathcal{P}\in \Pic(A\times A^{\vee})$ is the Poincar\'e bundle.
Künnemann has shown \cite[Theorem 5.2]{kunnemann-lefschetzdecompositionchowmotives} that intersecting with $\ell^{g-i}$ induces an isomorphism $\frak{h}^i(A) \rightarrow \frak{h}^{2g-i}(A)(g-i)$ for all $0\leq i \leq g$.
This induces a Lefschetz decomposition of the Chow--Künneth components $\frak{h}^i(A)$, see \cite[Theorem 5.1]{kunnemann-lefschetzdecompositionchowmotives}.
We are chiefly interested in the components $\frak{h}^3(A)$ and $\frak{h}^{2g-3}(A)$ when $g\geq 2$; in that case the Lefschetz decomposition has the form $\frak{h}^3(A) = \frak{h}^3_{\mathrm{prim}}(A) \oplus \ell \cdot \frak{h}^1(A)$ and $\frak{h}^{2g-3}(A) = \frak{h}^{2g-3}_{\mathrm{prim}}(A) \oplus \ell^{g-2}\cdot  \frak{h}^1(A)$. 
It has the property that the isomorphism $\ell^{g-3}\colon \frak{h}^3(A) \rightarrow \frak{h}^{2g-3}(A)(g-3)$ is a direct sum of isomorphisms $\frak{h}^3_{\mathrm{prim}}(A) \rightarrow \frak{h}^{2g-3}_{\mathrm{prim}}(A)$ and $\ell\cdot \frak{h}^1(A) \rightarrow \ell^{g-2} \frak{h}^1(A)$.
Taking Chow groups, we get a decomposition $\CH^{g-1}(\frak{h}^{2g-3}(A)) = \CH^{g-1}(\frak{h}_{\mathrm{prim}}^{2g-3}(A))\oplus \CH^{g-1}(\ell^{g-2} \cdot \frak{h}^1(A))$.

\subsection{Beauville components of $C$}
Consider a nice curve $C$ of genus $g\geq 2$ over $k$ with Jacobian $J$.
Let $e$ be a degree-$1$ divisor on $C$ and embed $C$ in $J$ using the Abel--Jacobi map based at $e$, sending $x\in C$ to the divisor class of $x-e$.
Decompose $[C] = [C]_0 + \cdots + [C]_{g-1}$ with $[C]_s \in \CH_{(s)}^{g-1}(J)$. 
In this subsection we analyze the component $[C]_1$ more closely, whose vanishing is equivalent to the vanishing of $\kappa(C)$.
Proposition \ref{proposition: pontryagin product C1 calculation} can be viewed as a determination of the `imprimitive part' of $[C]_1$.

For $\alpha \in \CH_p(J)$ and $\beta \in \CH_q(J)$, the Pontryagin product $\alpha\star \beta$ is the pushforward of $\alpha\times \beta$ under the addition map $J\times J\rightarrow J$.  If $n$ is a positive integer let $\alpha^{\star n}$ be the $n$-fold Pontryagin product of $\alpha$ with itself. 
Let $K_C\in \CH_0(C)$ denote the canonical divisor class and let $x_e\in J(k)$ be the point corresponding to the degree-$0$ divisor class $[(2g-2)e]-K_C$.

\begin{proposition}\label{proposition: pontryagin product C1 calculation}
    We have an equality in $\CH_{(1)}^1(J)$:
    \begin{align}\label{equation: pontryagin product calculation}
        (2g-2) \cdot [C]_0^{\star(g-2)} \star [C]_1 = [C]^{\star(g-1)} \star ([0]-[x_e]).
    \end{align}
    Here we view $[0]-[x_e]$ as an element of $\CH_0(J)$.
\end{proposition}
\begin{proof}
    If $D$ is a degree $g-1$ divisor class on $C$, let $\Theta_{D}$ be the image of the map $\Sym^{g-1}(C) \rightarrow J$ defined by $x\mapsto [x]-D$.
    Then $[C]^{\star (g-1)} = (g-1)![\Theta_{(g-1)e}]$.
    By Riemann--Roch, 
    \[(-1)_*[\Theta_{(g-1)e}] = [\Theta_{K_C-(g-1)e}] =[x_e] \star [\Theta_{(g-1)e}].\]
    Combining the last two sentences shows that $(-1)_*([C]^{\star (g-1)}) = [x_e]\star [C]^{\star(g-1)}$.
    
    Since Pontryagin product sends $\CH^{g-p}_{(s)}(J) \times \CH^{g-q}_{(t)}(J)$ to $\CH^{g-p-q}_{(s+t)}(J)$ and since $\CH^1_{(s)}(J)\neq 0$ only if $s\in\{0,1\}$, we calculate that $[C]^{\star(g-1)} = [C]_0^{\star(g-1)} + (g-1) \cdot [C]_0^{\star(g-2)} \star [C]_1$.
    Applying $(-1)_*$ and $[x_e]\star$ to the previous identity, we obtain 
    \begin{align*}
        (-1)_*[C]^{\star(g-1)} &= [C]_0^{\star(g-1)} - (g-1) \cdot [C]_0^{\star(g-2)} \star [C]_1, \\
        [x_e]\star [C]^{\star(g-1)} &= [x_e]\star [C]_0^{\star(g-1)} + (g-1) \cdot [C]_0^{\star(g-2)} \star [C]_1. 
    \end{align*}
    Note that $[x_e]\star$ acts trivially on $\CH^1_{(1)}(J)$ since $([x_e]-[0])\in \oplus_{s\geq 1} \CH^g_{(s)}(J)$ by the explicit description of the Beauville decomposition for zero-cycles \cite[bottom of p.\ 649]{Beauville-surlanneadechow}.
    Equating the right hand sides of the two centered equations proves that $(2g-2) \cdot [C]_0^{\star(g-2)} \star [C]_1 = [C]_0^{\star(g-1)} \star ([0]-[x_e])$.
    Since $[C]_s^{\star(g-1)} \star ([0]-[x_e]) \in \oplus_{t\geq s+1} \CH^1_{(t)}(J) = \{0\}$ for all $s\geq 1$, it follows that $[C]_0^{\star(g-1)} \star ([0]-[x_e]) = [C]^{\star(g-1)} \star ([0]-[x_e])$, concluding the proof.
\end{proof}

\begin{corollary}\label{corollary: C1 vanishes implies (2g-2)e canonical}
    Suppose that $[C]_1=0$ in $\CH_1(J)_{\Q}$.
    Then $(2g-2)e=K_C$ in $\CH_0(C)_{\Q}$.
\end{corollary}
\begin{proof}
    If $[C]_1=0$, then $[C]^{\star(g-1)} \star([0]-[x_e])=0$ by \eqref{equation: pontryagin product calculation}. 
    On the other hand, $[C]^{\star (g-1)} = (g-1)![\Theta_{(g-1)e}]$ is multiple of a theta divisor, in the notation of the proof of Proposition \ref{proposition: pontryagin product C1 calculation}.
    Since $\Theta_{(g-1)e}$ defines a principal polarization, the map $x\mapsto [\Theta_{(g-1)e + x}] - [\Theta_{(g-1
    )e}] = [\Theta_{(g-1)e}]\star ([x]-[0])$ induces an isomorphism $\varphi\colon J(k)\rightarrow \CH^1(J)_{\hom}$.
    Since $\varphi(x_e)$ is torsion, it follows that $x_e\in J(k)$ is itself torsion, as desired.
\end{proof}

The principal polarization defines an ample class $\ell\in \CH_{(1)}^1(J)$, which induces a Lefschetz decomposition $\frak{h}^{2g-3}(J) = \frak{h}^{2g-3}_{\mathrm{prim}}(J) \oplus \ell^{g-2}\cdot \frak{h}^1(J)$ as in \S\ref{subsec: lefschetz decomposition abelian varieties}.

\begin{corollary}\label{corollary: C1 is primitive}
    Suppose that $(2g-2)e = K_C$ in $\CH_0(C)_{\Q}$.
    Then $[C]_1 \in \CH^{g-1}(\frak{h}^{2g-3}_{\mathrm{prim}}(J))$.
\end{corollary}
\begin{proof}
    By definition and the discussion in \S\ref{subsec: chow kunneth decomposition abelian vars}, $[C]_1\in \CH_{(1)}^{g-1}(J) = \CH^{g-1}(\frak{h}^{2g-3}(J))$.
    Since $x_e$ is torsion, $(n)_*([0]-[x_e]) =0$ for some integer $n\geq 1$.
    The decomposition \eqref{equation: beauville decomposition} implies that $[0] = [x_e]$ in $\CH_0(J)_{\Q}$.
    Therefore \eqref{equation: pontryagin product calculation} shows that $[C]_0^{\star(g-2)} \star [C]_1=0$.
    Using properties of the $\mathfrak{sl}_2$-action on $\CH^*(J)_{\Q}$ (in the sense of \cite[\S1.3]{moonen-relationsbetweentautologicalcycles}), this implies that $\ell \cdot [C]_1 = 0$.
    Since $\CH^{g-1}(\frak{h}^{2g-3}_{\mathrm{prim}}(J))$ equals the kernel of $\ell \cdot (-) \colon \CH^{g-1}(\frak{h}^{2g-3}(J)) \rightarrow \CH^{g-1}(\frak{h}^{2g-1}(J)(1))$, the corollary follows.
\end{proof}

\subsection{Ceresa cycles}\label{subsec: ceresa cycles}
Let $C/k$ be a nice curve of genus $g\geq 2$ with Jacobian $J$.
Let $e$ be a degree-$1$ divisor on $C$ and let $\iota_e\colon C\rightarrow J$ be the Abel--Jacobi map based at $e$.
We define $\kappa_{C,e}\in \CH_1(J)$ using the formula \eqref{equation: ceresa cycle} from the introduction.
Using the Beauville decomposition \eqref{equation: beauville decomposition} to write $[C] = [\iota_e(C)] = \sum_{s=0}^{g-1} [C]_s$ with $[C]_s\in \CH^{g-1}_{(s)}(J)$, we calculate that 
\begin{align}\label{equation: ceresa cycle beauville components}
\kappa_{C,e}=\kappa(C) = 2[C]_1 + 2[C]_3 + \dots + 2[C]_{2\lfloor \frac{g-2}{2}\rfloor+1}
\end{align}
in $\CH_1(J)_{\Q}$.

\begin{lemma}
    If $\kappa_{C,e}$ is torsion, then $(2g-2)e-K_C$ is torsion.
\end{lemma}
\begin{proof}
    If $\kappa_{C,e}$ is torsion, then $[C]_1 = 0$ by \eqref{equation: ceresa cycle beauville components}.
    We conclude using Corollary \ref{corollary: C1 vanishes implies (2g-2)e canonical}.
\end{proof}

\begin{lemma}\label{lemma: torsion difference gives equivalent ceresa}
    If $e,e'\in \CH_0(C)$ are degree-$1$ divisors such that $e-e'$ is torsion, then $[\iota_e(C)] = [\iota_{e'}(C)]$ and $\kappa_{C,e} = \kappa_{C,e'}$ in $\CH_1(J)_{\Q}$.
\end{lemma}
\begin{proof}
    Suppose $e-e'$ has order $n$ in $\CH_0(C)$.
    Then $(n)\circ \iota_e = (n) \circ \iota_{e'}$, hence $(n)_*([\iota_e(C)] - [\iota_{e'}(C)]) = 0$ in $\CH_1(J)$.
    On the other hand, the decomposition \eqref{equation: beauville decomposition} shows that $(n)_*\colon \CH_1(J)_{\Q} \rightarrow \CH_1(J)_{\Q}$ is an isomorphism.
    Therefore $[\iota_e(C)] - [\iota_{e'}(C)]$ is torsion. 
    Hence $\kappa_{C,e}- \kappa_{C,e'}$ is torsion too.
\end{proof}
Let $\kappa(C)$ be the image of $\kappa_{C,e}$ in $\CH_1(J)_\Q$ for any choice of degree-$1$ divisor $e$ on $C$ such that $(2g-2)e = K_C$ in $\CH_0(C)_{\Q}$.
Lemma \ref{lemma: torsion difference gives equivalent ceresa} shows that this class is independent of the choice of $e$.
(If $k$ is not algebraically closed and no such $e$ exists over $k$, there exists a unique class $\kappa(C)\in \CH_1(J)_{\Q}$ such that $\kappa(C)_{\bar{k}} = \kappa(C_{\bar{k}})$ in $\CH_1(J_{\bar{k}})_{\Q}$, since Chow groups with $\Q$-coefficients satisfy Galois descent.)
We let $\bar{\kappa}(C)$ be the image of $\kappa(C)$ in $\Gr_1(J)_{\Q}$.
Since all degree-$1$ divisors $e$ on $C$ are algebraically equivalent, $\bar{\kappa}(C)$ is also the image of $\kappa_{e}(C)$ in $\Gr_1(J)_\Q$ for \emph{any} degree-$1$ divisor, not necessarily with the property that $(2g-2)e = K_C$ in $\CH_0(C)_{\Q}$.

Suppose now that $(2g-2)e= K_C$ in $\CH_0(C)_{\Q}$.
Since $[\iota_e(C)]\in \CH_1(J)_{\Q}$ is independent of the choice of $e$, the same is true for the classes $[C]_{s}$.
In particular, they are $\Aut(C)$-invariant.

\begin{proposition}\label{proposition: vanishing ceresa equivalent to vanishing C1}
    In the above notation, $\kappa(C)=0$ in $\CH_1(J)_{\Q}$ if and only if $[C]_1=0$ in $\CH_1(J)_{\Q}$ if and only if $[C]_s=0$ for all $s\geq 1$.
    Moreover $\bar{\kappa}(C) =0$ in $\Gr_1(J)_{\Q}$ if and only if $[C]_1 \in  \Gr_1(J)_{\Q}$ if and only if $[C]_s=0$ in $\Gr_1(J)_{\Q}$ for all $s\geq 1$.
\end{proposition}
\begin{proof}
    Using the expression \eqref{equation: ceresa cycle beauville components}, it suffices to prove $[C]_1 = 0$ implies $[C]_s=0$ for all $s\geq 1$.
    This follows from the third centered equation of \cite[Theorem 1.5.5]{ShouwuZhang}.
    The proof for $\bar{\kappa}(C)$ is identical.
\end{proof}

Proposition \ref{proposition: vanishing ceresa equivalent to vanishing C1} shows that the vanishing of $\kappa(C)$ is equivalent to the vanishing of $[C]_1$.
The next lemma isolates a summand of the Chow group that contains $[C]_1$.
Recall that if $M$ is a direct summand of $\frak{h}(J)$, then $\CH^p(M)$ is naturally a summand of $\CH^p(\frak{h}(J)) = \CH^p(J)_{\Q}$.

\begin{proposition}\label{proposition: C1 lies in specific submotive}
    In the above notation, $[C]_1\in \CH^{g-1}(\frak{h}^{2g-3}_{\mathrm{prim}}(J)^{\Aut(C)})\subset \CH_1(J)_{\Q}$.
\end{proposition}
\begin{proof}
    Combine Corollary \ref{corollary: C1 is primitive} and the fact that $[C]_1$ is $\Aut(C)$-invariant.
\end{proof}

\section{Vanishing criteria in the Chow and Griffiths groups}\label{sec: proofs theorems A and B}

The proof of Theorem \ref{thm: main} is quite short, but uses in a crucial way Kimura's notion of finite dimensional Chow motives \cite{Kimura}.

\begin{proof}[Proof of Theorem $\ref{thm: main}$]
 We use the definitions and notations of \S\ref{subsec: chow motives} and \S\ref{subsec: ceresa cycles}.
Choose a degree-$1$ divisor $e$ such that $(2g-2)e=K_C$ in $\CH_0(C)_{\Q}$ and decompose $[\iota_e(C)] = \sum_{s=-1}^{g-1} [C]_s$ with $[C]_s \in \CH^{g-1}_{(s)}(J)$.
Let $G = \Aut(C)$.
Proposition \ref{proposition: C1 lies in specific submotive} shows that $[C]_1\in \CH^{g-1}(\frak{h}^{2g-3}_{\mathrm{prim}}(J)^{G})$.

By the hypothesis and the Lefschetz isomorphism, $\HH^*(\frak{h}_{\mathrm{prim}}^{2g-3}(J)^{G}) = \HH^{2g-3}_{\mathrm{prim}}(J)^G=0$. Since any summand of the motive of an abelian variety is finite-dimensional in the sense of Kimura \cite[Example 9.1]{Kimura}, it follows from Kimura's \cite[Corollary 7.3]{Kimura} that $\frak{h}^{2g-3}_{\mathrm{prim}}(J)^{G} = 0$ and hence $[C]_1 = 0$ in $\CH^{g-1}(J)_{\Q}$.  
By Proposition \ref{proposition: vanishing ceresa equivalent to vanishing C1}, we conclude that $\kappa(C)=0$.  
\end{proof}

\begin{example}\label{example: theorem A genus 3 examples}
{\em 
    There is exactly one non-hyperelliptic genus 3 curve over $\C$ for which the criterion of \cite{QiuZhang} applies, namely the curve $y^3 = x^4 + 1$. The curve $y^3 = x^4 + x$ satisfies the weaker hypothesis of Theorem \ref{thm: main} (as was observed in \cite{Beauville}), so we deduce that $\kappa(C) = 0$ for this curve as well.  Beauville and Schoen studied the specific geometry of this curve and showed that $\bar{\kappa}(C) = 0$ \cite{BeauvilleSchoen}. 
}
\end{example}
\begin{example}
    {\em 
The genus $4$ curve $y^3 = x^5 + 1$ satisfies $\HH^3(J)^{\Aut(C)} = 0$ \cite[proof of Theorem 3.3]{LilienfeldtShnidman}, so $\kappa(C)=0$. 
}
\end{example}




Theorem \ref{thm: griffiths vanishing} will follow from the next proposition.
In that proposition and its proof, if $X$ is a nice variety over $\C$ we write $\HH^*(X)$ for the singular cohomology of $X(\C)$ with $\Q$-coefficients, seen as an object in the category of Hodge structures.

\begin{proposition}\label{prop: preparation griffiths vanishing}
    Let $C$ be a smooth, projective, integral curve over $\C$ with Jacobian $J$ and let $G\subset \Aut(C)$ be a subgroup with $\HH^0(J, \Omega_J^3)^{G}=0$.
    Then there exists an abelian variety $A/\C$ such that $\HH^3(J)^{G} \simeq \HH^1(A)(-1)$.
    If the Hodge conjecture holds for $J\times A$, then $\frak{h}^3(J)^{G}\simeq \frak{h}^1(A)(-1)$ and $\bar{\kappa}(C)=0$.
\end{proposition}
\begin{proof}
    Let $N^1\HH^3(J)$ be the largest sub-Hodge structure of $\HH^3(J)$ of type $(1,2)+(2,1)$.
    The polarization on $\HH^1(J)$ induces a polarization on $N^1\HH^3(J)$ and so $N^1\HH^3(J) \simeq \HH^1(B)(-1)$ for some abelian variety $B/\C$.
    The assumptions and the Hodge decomposition imply that $\HH^3(J)^G$ is a Hodge structure of type $(1,2) + (2,1)$. It follows that $\HH^3(J)^G \subset N^1\HH^3(J)$, so there exists an abelian subvariety $A\subset B$ with $\HH^3(J)^G \simeq \HH^1(A)(-1)$.

    We now show that the Hodge conjecture for $J\times A$ implies the claims of the final sentence.
    Fix mutually inverse isomorphisms of Hodge structures $\phi\colon \HH^3(J)^G\rightarrow \HH^1(A)(-1)$ and $\psi\colon \HH^1(A)(-1)\rightarrow \HH^3(J)^G$.
    By the Hodge conjecture, there exist morphisms of motives $\Phi\colon \frak{h}^3(J)^G \rightarrow \frak{h}^1(A)(-1)$ and $\Psi \colon \frak{h}^1(A)(-1) \rightarrow \frak{h}^3(J)^G$ (in other words, cycles on $J\times A$ with certain properties) with $\HH^*(\Phi) = \phi$ and $\HH^*(\Psi) = \psi$.
    Since $\frak{h}^1(A)(-1)$ and $\frak{h}^3(J)^G$ are Kimura finite-dimensional \cite[Example 9.1]{Kimura} and $\HH^*(\Phi) \circ \HH^*(\Psi)$ and $\HH^*(\Psi) \circ \HH^*(\Phi)$ are the identity, it follows from \cite[Proposition 7.2(ii)]{Kimura} (see also \cite[Corollaire 3.16]{andre-motifsdimensionfinie}) that $\Psi\circ \Phi$ and $\Phi\circ \Psi$ are themselves isomorphisms.
    Hence $\Phi$ and $\Psi$ are isomorphisms too and we conclude that $\frak{h}^3(J)^G \simeq \frak{h}^1(A)(-1)$.

    Using the Lefschetz isomorphism $\frak{h}^{2g-3}(J) \simeq \frak{h}^3(J)(-g+3)$ of \cite[Theorem 5.2]{kunnemann-lefschetzdecompositionchowmotives}, we obtain an isomorphism $\frak{h}^{2g-3}(J)^G\simeq \frak{h}^1(A)(-g+2)$.
    Similarly to the proof of Theorem \ref{thm: main}, we decompose $[\iota_e(C)] = \sum_{s=0}^{g-1} [C]_s$ with $[C]_s \in \CH^{g-1}_{(s)}(J)$ and observe that the class $[C]_1$ lies in $\CH_{(1)}^{g-1}(J)^G = \CH^{g-1}(\frak{h}^{2g-3}(J)^G)$.
    Lemma \ref{lemma: Griffiths group vanishes for h^1} combined with the isomorphism $\frak{h}^{2g-3}(J)^G\simeq \frak{h}^1(A)(-g+2)$ shows that $\CH^{g-1}(\frak{h}^{2g-3}(J)^G) = \CH^{g-1}(\frak{h}^{2g-3}(J)^G)_{\mathrm{alg}}$, hence every element of $\CH^{g-1}_{(1)}(J)^G$ lies in $\CH^{g-1}(J)_{\mathrm{alg},\Q}$.
    Therefore $[C]_1\in \CH^{g-1}(J)_{\mathrm{alg},\Q}$, hence the image of $[C]_1$ in $\Gr_1(J)_{\Q}$ vanishes.
    Proposition \ref{proposition: vanishing ceresa equivalent to vanishing C1} then implies that $\bar{\kappa}(C)$ vanishes too.
\end{proof}

\begin{proof}[Proof of Theorem $\ref{thm: griffiths vanishing}$]
    Since $C$ can be defined over a countable field and since such a field can be embedded in $\C$, Lemma \ref{lemma: base extension chow groups} shows that we may assume $k=\C$. We conclude by Proposition \ref{prop: preparation griffiths vanishing} applied to $G = \Aut(C)$.
\end{proof}



We now describe some families of curves for which the criterion of Theorem \ref{thm: griffiths vanishing} applies.
If $C/k$ is a curve we write $V(C) = \HH^0(C, \Omega_C)$, which we think of as an $\Aut(C)$-representation. Recall that it is $\bigwedge^3 V(C) \simeq \HH^0(J, \Omega_J^3)$ that appears in Theorem \ref{thm: griffiths vanishing}.
\begin{example}\label{example: Picard curves}
{\em 
 Let $C \colon y^3 = x^4 + ax^2 + bx + c$ be a Picard curve. Then $V(C) \simeq \chi \oplus \chi \oplus \chi^2$ as $C_3$-representations, where $\chi$ is a character of order $3$. It follows that $(\bigwedge^3V(C))^{\Aut(C)} = 0$, so that the condition of Theorem \ref{thm: griffiths vanishing} is satisfied. We consider these curves in detail in the next section.
 }
\end{example}


\begin{example}\label{ex: genus 4 mu3 family}
{\em 
The general $\mu_3$-cover of $\P^1$ with equation
    \[y^3 = x^2(x-1)^2(x^3 + ax^2 + bx + c),\]
    has genus $4$ and by \cite[Lemma 2.7]{Moonen-special} we have $V(C) \simeq \chi \oplus \chi \oplus \chi^2 \oplus \chi^2$ as $C_3$-representations. It follows that $(\bigwedge^3V(C))^{C_3} = 0$ and the abelian variety $A$ from Proposition \ref{prop: preparation griffiths vanishing} is 4-dimensional. If the Hodge conjecture holds for the $8$-dimensional $J \times A$, then $\bar{\kappa}(C) = 0$.
}
\end{example}

For our next class of examples, we consider one-parameter families of smooth projective curves with affine model of the form 
\[C_t \colon y^m = \left(\frac{x+1}{x-1}\right)^a \left(\frac{x+t}{x-t}\right)^b,\]
for integers $m,a,$ and $b$ satisfying $0 < a < b < m/2$ and $\gcd(m,a,b) = 1$. If $t\neq 0,\pm 1$, then $C_t$ is smooth projective of genus 
\[g = m +1 - \gcd(a,m) - \gcd(b,m),\]
which we will assume below is at least $3$. These curves admit an action by the dihedral group $D_m$, generated by the automorphisms $(x,y) \mapsto (x,\zeta_m y)$ and $\tau(x,y) = (-x, y^{-1})$.

\begin{lemma}
The induced action of $\mu_m \subset D_m$ on $V(C_t)$ has character $\sum_{n = 1}^{m-1} \epsilon(n) \chi^n$, where $\chi^n$ is the $1$-dimensional representation of $\mu_m$ on which $\zeta \in \mu_m$ acts by $\zeta^n$ and 
    \[\epsilon(n) = 
    \begin{cases}
        1 & \mbox{ if neither $m\nmid na$ nor $m \nmid nb$}\\
        0 & \mbox{otherwise}
    \end{cases}
    \]
 
\end{lemma}
\begin{proof}
This follows from \cite[Lemma 2.7]{Moonen-special}. 
\end{proof}

\begin{corollary}
    The space $\bigwedge^3 V(C_t)$ has a non-trivial $D_m$-invariant subspace if and only if there exists a triple of integers $0 < n_1 < n_2 < n_3 < m$ such that for all $i \in \{1,2,3\}$  neither $m\mid n_ia$ nor $m \mid n_i b$, and such that $n_1 + n_2 + n_3 = m$.
\end{corollary}
\begin{proof}
    By the lemma, a $\mu_m$-invariant subspace exists if and only if there exists a triple with the above divisibility properties and such that $n_1 +n_2 + n_3 \in \{m, 2m\}$. However, if there exists such a triple summing to $2m$ then the triple $m-n_3,m-n_2,m-n_1$ is another such triple whose sum is $m$. Finally, observe that $\tau$ swaps the $\chi^n$-isotypic component of $V(C_t)$ with the $\chi^{-n}$-isotypic component. It follows that $(\wedge^3V(C_t))^{\mu_m} = 0$ if and only if $(\wedge^3V(C_t))^{D_m} = 0$. 
\end{proof}

Below, we determine the triples $a < b < m/2$ such that $C_t$ satisfies the hypotheses of Theorem \ref{thm: griffiths vanishing}. 

\begin{example}\label{example: genus 4 mu5 cover}
{\em 
Suppose $\gcd(m,6) = 1$. As long as $m \geq 7$, the triple $1 < 2 < m-3$ corresponds to an invariant subspace of $\bigwedge^3V(C_t)$, so that Theorem \ref{thm: griffiths vanishing} does not apply.  However, when $m = 5$, the sum of any three nonzero elements of $\Z/5\Z$ is nonzero and so we find a genus $4$ family which does satisfy the criteria of Theorem \ref{thm: griffiths vanishing}, namely:
    \[y^5 = \left(\frac{x+1}{x-1}\right) \left(\frac{x+t}{x-t}\right)^2.\]
We compute that $\HH^3(J)^{\mu_5}$ is isomorphic to $\HH^1(A)(-1)$, for some abelian surface $A$ (using Proposition \ref{prop: preparation griffiths vanishing}). 
    By the Kani-Rosen formula \cite[Theorem B]{KaniRosen}, the Jacobian $J_t = \Jac(C_t)$ is isogenous to the square of the abelian surface $\Jac(C_t/\tau)$.  The Hodge conjecture is known for products of abelian surfaces \cite[Theorem 3.15]{Mari-hodgeconjectureproductsurfaces}, so $\bar{\kappa}(C_t) = 0$ by Proposition \ref{prop: preparation griffiths vanishing}. 
    }
\end{example}

\begin{example}
{\em 
    If $m$ is even, then proceeding in a manner similar to the previous example, one checks that $(\wedge^3 V(C_t))^{D_m} = 0$ implies that $m = 6$ (giving a subfamily of Example \ref{ex: genus 4 mu3 family}) or $m = 12$. In the latter case, we find the family of genus $6$ curves
     \[y^{12} = \left(\frac{x+1}{x-1}\right)^3 \left(\frac{x+t}{x-t}\right)^4,\]
    which indeed satsifies the hypotheses of Theorem \ref{thm: griffiths vanishing}.  
    }
\end{example}

\begin{example}\label{ex: genus 6 and genus 8}
{\em 
    In the remaining case where $m$ is an odd multiple of $3$, we check that the condition $(\wedge^3 V(C_t))^{D_m} = 0$ implies that $m = 9$ or $m = 15$ and moreover that one of $a$ or $b$ is $m/3$. These correspond to the genus $6$ family  
 \begin{equation}\label{eq: genus 6}
 y^9 = \left(\frac{x+1}{x-1}\right) \left(\frac{x+t}{x-t}\right)^3,
 \end{equation}
 and the genus $8$ family
  \begin{equation}\label{eq: genus 8}
  y^{15} = \left(\frac{x+1}{x-1}\right)^3 \left(\frac{x+t}{x-t}\right)^5,
  \end{equation}
        which again satisfy the hypotheses of Theorem \ref{thm: griffiths vanishing}. 
    }
\end{example}



One can check (by specializing) that the families above are generically non-hyperelliptic, so these are genuinely interesting curves satisfying $\overline{\kappa}(C) = 0$ (in some cases, conditional on the Hodge conjecture). Quotienting by the involution $\tau$, we also recover the first known examples of families of curves $C$ with $\overline{\kappa}(C) = 0$ and $\Aut(C) = 0$:

\begin{theorem}\label{thm: trivial automorphism group example}
    Assuming the Hodge conjecture, then for $g \in \{3,4\}$, there exist non-trivial one-parameter families $\{X_u\}$ of genus $g$ curves over $\C$ with $\bar{\kappa}(X_u) = 0$ and which generically satisfy $\Aut(X_u) = \{1\}$. 
\end{theorem}

\begin{proof}
    Consider the family $C_t$ in $(\ref{eq: genus 6})$, which satisfies the hypotheses of Theorem \ref{thm: griffiths vanishing}.
    The quotient $Y_t = C_t/\tau$ is a family of genus $3$ curves. Since we assume the Hodge conjecture, Theorem \ref{thm: griffiths vanishing} implies that $\overline{\kappa}(C_t) = 0$ and hence $\overline{\kappa}(Y_t) = 0$ as well.  Computing in Magma \cite{Magma}, we find that $Y_t$ is isomorphic to the curve:
    \[X_u \colon y^4 = xy^3 + u x^3 + 3y^2 - 2u x - u y - 1,\]
 where $u = (t-1)/(t+1)$. Specializing to $u = 2$, we check in Magma that the automorphism group is trivial and hence $\Aut(X_u) = \{1\}$ generically.   

 The proof when $g = 4$ is similar, using the involution on the genus $8$ family $(\ref{eq: genus 8})$. 
\end{proof}

\begin{remark}\label{rem: costa}
{\em
    Edgar Costa has verified for us, using \cite{CostaEtAl}, that the generic endomorphism algebra of $\Jac(X_u)$ is contained in $K = \Q(\zeta_9 + \zeta_9^{-1})$; by \cite[\S3]{Ellenberg2001}, this inclusion is an  equality. Similarly, in the genus $4$ family, the algebra $\End^0(\Jac(X_u))$ contains $\Q(\zeta_{15} + \zeta_{15}^{-1})$.  
}    
\end{remark}

It would be interesting to study the families in Examples \ref{ex: genus 4 mu3 family} and \ref{example: genus 4 mu5 cover}-\ref{ex: genus 6 and genus 8} in more detail, and in particular to try to prove the relevant case of the Hodge conjecture and determine the locus where the Ceresa class vanishes in the Chow group (and not just the Griffiths group). We do this for the family of Picard curves, in the next section. 


\section{Picard curves}\label{sec: Picard}

\subsection{Generalities}\label{subsec: generalities}
Let $k$ be a field of characteristic zero. 
A Picard curve over $k$ is by definition a nice curve with an  affine model $y^3 = f(x) = x^4+ax^2+bx+c$ for some $a,b,c\in k$.
Conversely, given such a polynomial $f(x) \in k[x]$ of nonzero discriminant, the projective closure of $y^3= f(x)$ in $\P^2_k$ is a nice curve denoted by $C_f$.
It has a unique point at infinity $P_{\infty}$, which is $k$-rational.

For every third root of unity $\omega \in \bar{k}$, the map $(x,y)\mapsto (x,\omega y)$ defines an automorphism of $C_{f,\bar{k}}$. We view $\mu_3$ as a subgroup of $\Aut(C_{f,\bar{k}})$ in this way. Then $\mu_3$ also acts on the Jacobian $J_{f,\bar{k}}$ by taking images of divisors. 

The discriminant of $f$ has the following expression:
\begin{align}\label{equation: disc picard curve}
    \disc(f)=-4 a^3 b^2 - 27 b^4 + 16 a^4 c + 144 a b^2 c - 128 a^2 c^2 + 256 c^3.
\end{align}
We view $f$ as the dehomogenization $F(x,1)$ of the quartic form $F(X,Z) = X^4+aX^2Z^2+bXZ^3+cZ^4$, and we
define $I(f)$ and $J(f)$ to be the usual $I$- and $J$-invariants attached to $F$, as in \cite[\S2]{bhargavashankar-2selmer}. Their explicit formulae in our case are:
\begin{align*}
    I(f) &= a^2+12c, \\
    J(f) &= 72ac-2a^3-27b^2. 
\end{align*}
The 19th century invariant theorists observed the identity $J(f)^2 = 4I(f)^3 - 27\cdot \disc(f)$, which can be verified by direct computation. Therefore $P_f:=(I(f),J(f))$ is a $k$-point on the elliptic curve \[E_f\colon y^2 = 4x^3-27\cdot \disc(f).\]

\subsection{Ceresa vanishing criteria}\label{subsec: ceresa vanishing criteria}

Since $(2g-2)P_{\infty}=4P_{\infty}$ is canonical, we may use $P_{\infty}$ to embed $C_f$ in its Jacobian $J_f$ and define the Ceresa cycle $\kappa_{C_f, P_{\infty}}\in \CH_1(J_f)$ as in the introduction; we denote it by $\kappa_f$ for simplicity.
Recall that $\kappa(C_f)$ denotes the image of $\kappa_f$ in $\CH_1(J_f)_{\Q}$ and $\bar{\kappa}(C_f)$ its image in $\Gr_1(J_f)_{\Q}$. 
Theorem \ref{thm: Picard main} follows from the following slightly stronger theorems, whose proofs will take up the rest of this section.

\begin{theorem}\label{thrm main body: picard curve vanish griffiths group}
    There exists an integer $N\geq 1$ $($depending neither on $f$ nor $k)$ such that $N\cdot \kappa_f\in \CH^2(J_f)_{\mathrm{alg}}$ for every Picard curve $C_f$ over every algebraically closed field $k$ (of characteristic zero).
\end{theorem}

\begin{theorem}\label{thm: picard vanishing chow group}
    The Ceresa cycle $\kappa_f\in \CH_1(J)$ is torsion if and only if $P_f\in E_f(k)$ is torsion.
    Moreover, if $k$ is algebraically closed, there exists an integer $M\geq 1$ with the following property: if $C_f$ is a Picard curve and $\kappa_f$ is torsion, then $\ord(\kappa_f)$ divides $M\cdot \ord(P_f)$ and $\ord(P_f)$ divides $M\cdot \ord(\kappa_f)$.
\end{theorem}

Theorem \ref{thrm main body: picard curve vanish griffiths group} will be proven in \S\ref{subsec: weil classes and Schoens theorem}, and Theorem \ref{thm: picard vanishing chow group} will be proven in \S\ref{subsec: proof of the vanishing criterion picard chow group}.
A standard argument using Lemma \ref{lemma: base extension chow groups} shows that we may assume $k=\C$.
So in the remainder of \S\ref{sec: Picard}, all varieties will be over $\C$, and cohomology will be singular cohomology.

\begin{remark}
    {\em Theorem \ref{thm: picard vanishing chow group} generalizes \cite[Theorem 5.16]{LagaShnidman-CeresaBiellipticPicard}, which considered the special case where $b = 0$. There, we exploited the bielliptic cover to show that the Ceresa cycle maps via a correspondence to a multiple of the point $Q_f = (a^2 -4c,a(a^2 -4c))$ on the elliptic curve $E_f' \colon y^2 = x^3 +16 \cdot \disc(f)$. This is compatible with the general case since there is a $3$-isogeny $\phi_f \colon E'_f \rightarrow E_f$, and one checks using the explicit formula for $\phi_f$ \cite[Equation (2)]{BhargavaElkiesShnidman} that $\phi_f(Q_f) = P_f$.   }
\end{remark}
\begin{remark}
{\em
    Is it always the case that $\kappa_f \neq 0$? (Recall that $\kappa_f$ lies in the Chow group with $\Z$-coefficients.) We cannot conclude this from our proof of Theorem \ref{thm: picard vanishing chow group} below since we have worked with $\Q$-coefficients, and we make use of various isogenies whose degrees we do not control.
    }
\end{remark}

\subsection{Multilinear algebra}\label{subsec: multilinear algebra}

Our first goal (Proposition \ref{proposition: wedge3 iso to ell curve as hodge structures}) is to explicitly identify the abelian variety $A$ of Proposition \ref{prop: preparation griffiths vanishing} for Picard curves.

Write $\calO = \Z[\omega]$ for the ring of Eisenstein integers with $\omega^2+ \omega +1=0$ and let $K=\Q(\sqrt{-3})$ be its fraction field. 
Let $C$ be a Picard curve over $\C$ with Jacobian variety $J$.
The $\mu_3$-action on $C$ extends to an embedding $\calO \subset \End(J)$. 
Using this action, the singular cohomology group $\HH^1(J;\Z)$ is a free $\calO$-module of rank $3$, and $\HH^1(J;\Q)$ is a $3$-dimensional $K$-vector space. 
The next lemma says that the criterion of Theorem \ref{thm: griffiths vanishing} is always satisfied for Picard curves.

\begin{lemma}\label{lemma: mu3-invariants type (1,2)+(2,1)}
    $\HH^0(J, \Omega_J^3)^{\mu_3} = 0$ and the Hodge structure $\HH^3(J;\Q)^{\mu_3}$ is of type $(1,2) +(2,1)$. 
\end{lemma}
\begin{proof}
    Since $\HH^0(J, \Omega_J^3) \simeq \bigwedge^3 \HH^0(C, \Omega_C^1)$, the first claim follows from a calculation with differentials (Example \ref{example: Picard curves}). The second claim follows from the Hodge decomposition for $\HH^3(J;\Q)$.
\end{proof}

We may view $\HH^1(J;\Q)$ either as a $K$-vector space or $\Q$-vector space; when we perform tensor operations, we will add the subscript $K$ when we view it as a $K$-vector space, and add no subscript otherwise.
For example, cup product induces an isomorphism $\bigwedge^3 \HH^1(J;\Q) \simeq \HH^3(J;\Q)$, and we will use this identification without further mention.

The universal property of exterior powers induces a canonical $\Q$-linear surjection $\bigwedge^3 \HH^1(J;\Q) \rightarrow \bigwedge^3_K \HH^1(J;\Q)$.
It is well known (see \cite[Lemma 12(i)]{moonenzarhin-weilclasses} or \cite[Lemma 4.3]{Delignemilne-hodgecyclesmotives}) that this map admits a canonical splitting, which we use to view $\bigwedge_K^3 \HH^1(J;\Q)$ as a direct summand of $\HH^3(J;\Q)$.
\begin{lemma}\label{lem:wedge = invariants}
    We have $\HH^3(J;\Q)^{\mu_3}= \bigwedge_K^3 \HH^1(J;\Q)$ inside $\HH^3(J;\Q)$.
    Moreover $\dim_{\Q} \HH^3(J;\Q)^{\mu_3}=2$.
\end{lemma}
\begin{proof}
    It suffices to prove the statements after tensoring with $\C$.
    Let $g\in \mu_3\subset \Aut(C)$ be a nontrivial element. 
    The action of $g$ on $\HH^1(J;\Q)$ has eigenvalues (with multiplicity) $\omega, \omega, \omega, \omega^2, \omega^2, \omega^2$ so we can write $\HH^1(J;\C) = V_1 \oplus V_2$ where $V_i$ is the $\omega^i$-eigenspace. 
    Since a three element subset of these eigenvalues have product $1$ if and only if they are all equal, $\HH^3(J;\C)^{\mu_3} = (\bigwedge^3 V_1) \oplus (\bigwedge^3 V_2)$.
    On the other hand, an argument similar to the proof of \cite[Proposition 4.4]{Delignemilne-hodgecyclesmotives} shows $(\bigwedge^3_K \HH^1(J;\Q))\otimes_{K} \C = \bigwedge^3_{K\otimes \C} \HH^1(J;\C) = (\bigwedge^3 V_1) \oplus (\bigwedge^3 V_2)$, proving the equality. 
    The explicit description of this subspace shows that it is $2$-dimensional over $\Q$.
\end{proof}

Let $E$ be the elliptic curve with Weierstrass equation $y^2 = x^3+1$.

\begin{proposition}\label{proposition: wedge3 iso to ell curve as hodge structures} 
    There is an isomorphism of Hodge structures $\HH^3(J;\Q)^{\mu_3} \simeq \HH^1(E,\Q)(-1)$.
\end{proposition}
\begin{proof}
    Since there exists a unique elliptic curve up to isogeny with endomorphism algebra $K$, there exists a unique Hodge structure of dimension $2$, type $(0,1)+(1,0)$ and carrying an action of $K$.
    Since both $\HH^3(J;\Q)^{\mu_3}$ and $\HH^1(E;\Q)$ have these properties (the former by Lemma \ref{lem:wedge = invariants}), they must be isomorphic, and we conclude using Lemma \ref{lem:wedge = invariants}.
\end{proof}

\subsection{Weil classes and Schoen's theorem}\label{subsec: weil classes and Schoens theorem}

Our next goal is to upgrade the isomorphism 
of Proposition \ref{proposition: wedge3 iso to ell curve as hodge structures} 
to an isomorphism in the category of Chow motives, and hence deduce (using Proposition \ref{prop: preparation griffiths vanishing}) the vanishing of $\bar{\kappa}(C)$ in the Griffiths group.
We use the following special case of a result of Schoen, which crucially uses the assumption that the endomorphism algebra of $J\times E$ contains $K=\Q(\omega)$:
\begin{theorem}[Schoen]\label{thm: schoen}
    The Hodge conjecture holds for $J\times E$.
\end{theorem}
\begin{proof}
    Since $J\times E$ is four-dimensional, it suffices to prove Hodge classes in $\HH^4(J\times E;\Q)$ are algebraic.
    By \cite[Theorem (0.1), Part (i) and (iv)]{MoonenZarhin-hodgeclassesabelianvarietieslowdimension}, such Hodge classes are sums of products of divisor classes (which are algebraic) and Weil classes $W_K := \bigwedge^4_K \HH^4(J\times E;\Q) \subset \HH^4(J\times E;\Q)$.
    Since $K= \Q(\omega)$ and the embedding $K\subset \End(J\times E)\otimes \Q$ can be chosen to have signature $(2,2)$, Schoen has shown in \cite{schoen-addendumhodgeclassesselfproducts} that the classes in $W_K$ are algebraic, concluding that all Hodge classes of $\HH^4(J\times E;\Q)$ are algebraic.
\end{proof}

Recall from \S\ref{subsec: chow motives} our conventions on motives, the canonical Chow--K\"unneth components $\frak{h}^i(J)$ and $\frak{h}^j(E)$, and the motive of fixed points of a finite group action.

\begin{corollary}\label{cor: isom of motives}
    There is an isomorphism $\frak{h}^3(J)^{\mu_3} \simeq  \frak{h}^1(E)(-1)$.
\end{corollary}
\begin{proof}
    This follows from Proposition \ref{prop: preparation griffiths vanishing}, using Proposition \ref{proposition: wedge3 iso to ell curve as hodge structures} and Theorem \ref{thm: schoen}.
\end{proof}


\begin{corollary}\label{corollary: ceresa vanishes griffiths group picard curve}
    If $C$ is a Picard curve over $\C$, then $\bar{\kappa}(C)=0$. 
\end{corollary}
\begin{proof}
    This follows from Proposition \ref{prop: preparation griffiths vanishing}, using Proposition \ref{proposition: wedge3 iso to ell curve as hodge structures} and Theorem \ref{thm: schoen}.
\end{proof}

\begin{proof}[Proof of Theorem \ref{thrm main body: picard curve vanish griffiths group}]
    Consider the parameter space $S_0 = \{(a,b,c) \mid \disc(f)\neq 0\} \subset \A^3_{\Q}$ of Picard curves.
    Let $\eta$ be the generic point of $S_0$ with function field $k(\eta)= \Q(a,b,c)$, let $C_{\eta}$ be the generic Picard curve over $k(\eta)$ with Jacobian $J_{\eta}$, and let $\kappa_{C_{\eta}}\in \CH_1(J_{\eta})$ be the Ceresa cycle of $C_{\eta}$ based at the point at infinity.
    Fix an embedding $j\colon k(\eta) \rightarrow \C$.
    By Corollary \ref{corollary: ceresa vanishes griffiths group picard curve}, the base change of $\kappa_{C{\eta}}$ along $j$ is torsion in the Griffiths group.
    By Lemma \ref{lemma: base extension chow groups}, this implies $\kappa_{C_{\eta}}$ is itself torsion in $\Gr_1(J_{\eta})$.
    Let $N_0$ be the finite order of $\kappa_{C_{\eta}}$ in $\Gr_1(J_{\eta})$.
    By spreading out, it follows that $\kappa_{C_f}$ is $N_0$-torsion in $\Gr_1(J_f)$ for all $f$ in an open dense subset $U_0\subset S_0$.
    By applying the same argument to the generic points of the irreducible components of $S_0\setminus U_0$, there exists an open dense $U_1\subset S_0$ containing $U_0$ and those generic points such that $\kappa_{C_f}$ is $N_1$-torsion in $\Gr_1(J_f)$ for all $f$ in $U_1$.
    Repeating this process, we obtain a sequence of open subsets $U_0\subset U_1\subset \cdots \subset S_0$ whose complements have strictly decreasing codimension. 
    Therefore this sequence must terminate after finitely many steps hence there exists an integer $N$ such that $\kappa_{C_f}$ is $N$-torsion in $\Gr_1(J_f)$ for all $f\in S_0(k)$ over every algebraically closed field of characteristic zero.
\end{proof}

The remainder of the section is devoted to proving Theorem \ref{thm: picard vanishing chow group}. 
To this end, we will analyze the Abel--Jacobi image of $\kappa_f$ (the ``normal function'' associated to $\kappa_f$) in the next two subsections.

\subsection{The Abel--Jacobi map}\label{subsec: abel-jacobi map}
If $X$ is a smooth variety over $\C$, we will use the notion of an integral (respectively rational) variation of (pure) Hodge structures over the complex manifold $X(\C)$, called a $\Z$-VHS (respectively $\Q$-VHS) for short; see \cite[\S5.3.1]{voisin-hodgeii} for definitions.
If $V$ is a Hodge structure of odd weight $2k-1$, its intermediate Jacobian $\mathrm{J}(V)$ is the complex torus
\begin{align*}
    (V\otimes_{\Z} \C)/(F^k + V_{\tau}),
\end{align*}
where $F^k\subset V\otimes_{\Z} \C$ is a part of the descending Hodge filtration and $V_{\tau}$ denotes the quotient of $V$ by its torsion subgroup.
More generally, if $\mathbb{V}$ is a $\Z$-VHS of weight $2k-1$ over $X(\C)$, we can define its intermediate Jacobian $\mathrm{J}(\mathbb{V})\rightarrow X(\C)$, a relative complex torus whose fibers over points $x\in X(\C)$ are the classical intermediate Jacobians $\mathrm{J}(\mathbb{V}_x)$, see \cite[\S7.1.1]{voisin-hodgeii}.
If $\mathbb{V}, \mathbb{W}$ are two $\Z$-VHS of odd weight, then $\mathrm{J}(\mathbb{V}(p)) = \mathrm{J}(\mathbb{V})$ for all $p \in \Z$ and a morphism $\mathbb{V} \rightarrow \mathbb{W}$ of $\Z$-VHS induces a homomorphism of (relative) complex tori $\mathrm{J}(\mathbb{V}) \rightarrow \mathrm{J}(\mathbb{W})$.

If $X/\C$ is a nice variety and $0\leq p\leq \dim(X)$, we write $\mathrm{J}^p(X) = \mathrm{J}(\HH^{2p-1}(X(\C);\Z))$.
In this situation there is an Abel--Jacobi map 
\begin{align}
    \mathrm{AJ}^p_{X}\colon \CH^p(X)_{\hom} \rightarrow \mathrm{J}^p(X),
\end{align}
defined in \cite[\S7.2.1]{voisin-hodgeii}.
Moreover, if $S/\C$ is a smooth variety, $\pi\colon X\rightarrow S$ a smooth projective morphism with geometrically integral fibres and $p \in \Z_{\geq 0}$, then $R^p \pi_*\Z$ (pushforward of the constant sheaf in the analytic topology) has the structure of a $\Z$-VHS over $S(\C)$.
If $Z$ is a codimension $p$ cycle on $X$ all of whose components are flat over $S$, and such that $Z_s \in \CH^p(X_s)_{\hom}$ for every $s\in S(\C)$, then Griffiths has shown that there exists a holomorphic section $\AJ(Z)$ of the relative complex torus $\mathrm{J}(R^p\pi_* \Z) \rightarrow S(\C)$ with the property that $\AJ(Z)_s = \AJ_X^p(Z_s)$ for all $s\in S(\C)$; this is called the normal function associated to $Z$.

We record the fact that Abel--Jacobi maps are compatible with correspondences.
Let $X,Y$ be nice varieties over $\C$ and let $\gamma \in\CH^{r+\dim(X)}(X\times Y)$ be a correspondence of degree $r$.
This induces for every $p\geq 0$ a homomorphism $\gamma_* \colon \CH^p(X) \rightarrow \CH^{p+r}(Y)$ via the formula $\alpha \mapsto \pi_{Y,*}(\pi_{X}^*(\alpha) \cdot \gamma)$, where $\pi_X\colon X\times Y\rightarrow X$ and $\pi_Y\colon X\times Y\rightarrow Y$ denote the projections.
The same formula defines morphism of Hodge structures $\HH^p(X;\Z) \rightarrow \HH^{p+2r}(Y)(r)$ for every $p$, hence a homomorphism of complex tori $\gamma_*\colon \mathrm{J}^p(X) \rightarrow \mathrm{J}^{p+r}(Y)$. 

\begin{lemma}\label{lemma: abel-jacobi compatible with correspondences}
    In the above notation, $\gamma_*$ sends $\CH^p(X)_{\hom}$ to $\CH^{p+r}(Y)_{\hom}$.
    Moreover for every $\alpha \in \CH^p(X)_{\hom}$, $\gamma_*(\AJ_X^p(\alpha)) = \AJ_Y^{p+r}(\gamma_*(\alpha))$.
\end{lemma}
\begin{proof}
    The first sentence follows from the compatibility of the cycle class map with correspondences \cite[Proposition 9.21]{voisin-hodgeii}.
    To prove the compatibility of the Abel--Jacobi map with correspondences, it suffices to prove the compatibility with pullbacks, pushforwards and intersection product. The case of pullback is elementary, using the definition of $\AJ^p_X$ in terms of extensions of Hodge structures, see \cite[\S2.2]{Charles-zerolocusnormalfunctions}.
    The case of pushforward follows from that of pullback and Poincare duality. 
    Finally, compatibility with intersection product follows from \cite[Proposition 9.23]{voisin-hodgeii}.
\end{proof}

Let $S = \{(a,b,c) \mid \disc(f) \neq 0\}\subset \A^3_{\C}$ be the parameter space of Picard curves over $\C$.
We will identify $\C$-valued points of $S$ with polynomials $f = x^4+ax^2+bx+c \in \C[x]$ of nonzero discriminant.
Let $\mathcal{C}\rightarrow S$ be the universal Picard curve, and let $\pi\colon \mathcal{J}\rightarrow S$ its relative Jacobian variety.
The point at infinity defines a section $P_{\infty}$ of $\mathcal{C}$.
Let $\kappa_{\mathcal{C}} \in \CH_1(\mathcal{J})$ be the universal Ceresa cycle with respect to this section.
(Comparing with our earlier notation, we have $\mathcal{C}_f = C_f$ and $\kappa_{\mathcal{C},f} = \kappa_f$ for every $f \in S(\C)$.)

Let $\mathbb{V} = R^3 \pi_* \Z$ be the $\Z$-VHS on $S(\C)$ interpolating the cohomology groups $\HH^3(\mathcal{J}_f;\Z)$ for $f\in S(\C)$.
Then the normal function $\AJ(\kappa_{\mathcal{C}})$ is a section of $\mathrm{J}(\mathbb{V}) \rightarrow S(\C)$ interpolating $\AJ_{J_f}^2(\kappa_f)$.

\begin{proposition}\label{proposition: ceresa torsion iff AJ image torsion}
    For $f\in S(\C)$, the class $\kappa_f\in \CH^2(J_f)_{\mathrm{hom}}$ is torsion if and only if $\AJ_{J_f}^2(\kappa_f)\in \mathrm{J}^2(J_f)$ is torsion. 
    The torsion order of $\kappa_f$, if finite, equals the torsion order of $\AJ_{J_f}^2(\kappa_f)$.
\end{proposition}
\begin{proof}
Let $E$ be the elliptic curve $y^2 = x^3+1$.
Corollary \ref{cor: isom of motives} shows that there exists a correspondence $\gamma \in \CH^2(E\times J)$ such that $\gamma_*$ induces isomorphisms $\CH^1(E)_{\hom, \Q} \rightarrow \CH^2_{(1)}(J)^{\mu_3}$ and $\HH^1(E;\Q) \rightarrow \HH^3(J;\Q)^{\mu_3}(1)$.
Let $\varphi$ be the restriction of $\AJ_{J_f}^2\otimes \Q$ to $\CH_{(1)}^2(J_f)^{\mu_3}$.
By Lemma \ref{lemma: abel-jacobi compatible with correspondences} these form a commutative diagram:
\[\begin{tikzcd}
	{\CH^1(E)_{\hom,\Q}} & {\mathrm{J}^1(E)\otimes \Q} \\
	{\CH_{(1)}^2(J_f)^{\mu_3}} & {\mathrm{J}^2(J_f)^{\mu_3}\otimes \Q}
	\arrow["\AJ", from=1-1, to=1-2]
	\arrow["{\gamma_*}", from=1-1, to=2-1]
	\arrow["{\gamma_*}", from=1-2, to=2-2]
	\arrow["\varphi", from=2-1, to=2-2]
\end{tikzcd}\]
All arrows except $\varphi$ are isomorphisms of abelian groups.
Therefore $\varphi$ is an isomorphism too. 
Since the image of $\kappa_f$ in $\CH^2(J_f)_{\Q}$ lies in $\CH^2_{(1)}(J_f)^{\mu_3}$ by \eqref{equation: ceresa cycle beauville components}, we conclude that $\kappa_f$ is torsion if and only if $\AJ_{J_f}^2(\kappa_f)$ is.
The claim about torsion orders follows from the fact that $\AJ^2_{J_f}$ is injective on torsion subgroups by a result of Murre \cite[Theorem 10.3]{Murre-applicationsktheorycycles}.
\end{proof}

\subsection{Identifying the complex torus $\mathrm{J}(\mathbb{V}^{\mu_3})$
}

The $\mu_3$-action on $\mathcal{C}$ induces, via functoriality, a $\mu_3$-action on $\mathcal{J}$, $\mathbb{V}$ and $\mathrm{J}(\mathbb{V})$. 
The subsheaf of fixed points $\mathbb{V}^{\mu_3}$ has the structure of a $\Z$-VHS.
The connected component of the identity $\mathrm{J}(\mathbb{V})^{\mu_3,\circ}$ of $\mathrm{J}(\mathbb{V})^{\mu_3}$ is a relative complex torus over $S(\C)$. 
Moreover, the natural homomorphism $\mathrm{J}(\mathbb{V}^{\mu_3}) \rightarrow \mathrm{J}(\mathbb{V})$ induces an isomorphism onto $\mathrm{J}(\mathbb{V})^{\mu_3,\circ}$.
\begin{lemma}
    The multiple $3 \cdot \mathrm{AJ}(\kappa_{\mathcal{C}})$ lands in $\mathrm{J}(\mathbb{V})^{\mu_3,\circ}$.
\end{lemma}
\begin{proof}
    Since the Ceresa cycle $\kappa_{\mathcal{C}}$ is $\mu_3$-invariant, $\mathrm{AJ}(\kappa_{\mathcal{C}})$ lands in $\mathrm{J}(\mathbb{V})^{\mu_3}$. The norm map $\mathrm{N}\colon \mathrm{J}(\mathbb{V}) \rightarrow \mathrm{J}(\mathbb{V})$, defined by $x\mapsto x+ \omega \cdot x + \omega^2 \cdot x$, lands in $\mathrm{J}(\mathbb{V})^{\mu_3,\circ}$, since the image must be connected and $\mu_3$-invariant.
    We conclude that $\mathrm{N}(\mathrm{AJ}(\kappa_{\mathcal{C}})) = 3\cdot \mathrm{AJ}(\kappa_{\mathcal{C}})$ lands in $\mathrm{J}(\mathbb{V})^{\mu_3,\circ}$.
\end{proof}
Therefore $3\cdot \mathrm{AJ}(\kappa_{\mathcal{C}})$ defines a section of $\mathrm{J}(\mathbb{V})^{\mu_3,\circ}$, hence we may view it as a section of the complex torus $\mathrm{J}(\mathbb{V}^{\mu_3}) \rightarrow S(\C)$ in what follows.
We study this relative complex torus (up to isogeny) in the next two propositions.

\begin{proposition}\label{proposition: identification complex tori}
    Let $\mathcal{E} \rightarrow S$ be the relative elliptic curve with Weierstrass equation \[\mathcal{E}\colon y^2 = 4x^3 - 27 \cdot \disc(f).\]
    Then the relative complex torus $\mathrm{J}(\mathbb{V}^{\mu_3})\rightarrow S$ is isogenous to the relative complex torus $\mathcal{E}(\C) \rightarrow S(\C)$.
\end{proposition}
\begin{proof}
    Recall from \S\ref{subsec: multilinear algebra} that we write $\calO = \Z[\omega]$ and $K = \Q(\omega)$.
    The $\Z$-VHS $R^1 \pi_* \Z$ interpolating the cohomology groups $\HH^1(\mathcal{J}_f;\Z)$ comes equipped with an action of $\calO$, and cup product induces an isomorphism $\bigwedge^3 R^1 \pi_* \Z \simeq R^3 \pi_* \Z = \mathbb{V}$.
    Let $\mathbb{W} = \bigwedge^3_{\mathcal{O}} R^1\pi_* \Z$ (the third exterior product of $R^1\pi_*\Z$, viewed as a sheaf of $\calO$-modules), a $\Z$-VHS with an $\calO$-action.
    Lemma \ref{lem:wedge = invariants} shows that $\mathbb{W} \otimes \Q \simeq \mathbb{V}^{\mu_3} \otimes \Q$, so $\mathrm{J}(\mathbb{W})$ and $\mathrm{J}(\mathbb{V}^{\mu_3})$ are isogenous.

    To analyze $\mathrm{J}(\mathbb{W})$, we analyze the $\Z$-VHS $\mathbb{W}(1)$ more closely. 
    It has an $\calO$-action by construction, Lemma \ref{lem:wedge = invariants} shows that it has constant rank $2$, and Lemma \ref{lemma: mu3-invariants type (1,2)+(2,1)} shows that it has type $(1,0)+(0,1)$.
    Since $\calO$ has class number $1$, there exists a unique $\Z$-Hodge structure with these properties, hence $\mathbb{W}_f(1) \simeq \HH^1(E;\Z)$ for every $f\in S(\C)$, where $E$ is the elliptic curve with Weierstrass equation $y^2 = x^3+1$.

    Therefore $\mathrm{J}(\mathbb{W}) \rightarrow S(\C)$ is an isotrivial family of elliptic curves: analytically locally on $S(\C)$, it is isomorphic to $E(\C) \times S(\C) \rightarrow S(\C)$.
    It follows that the sheaf of local isomorphisms between $\mathrm{J}(\mathbb{W})$ and $E(\C) \times S(\C)$ is an $\Aut(E)$-torsor in the analytic topology on $S(\C)$.
    Since $\Aut(E) \simeq \mu_6$ is finite, this torsor is the analytification of an \'etale $\mu_6$-torsor on $S$.
    
    To analyze \'etale $\mu_6$-torsors on $S$, consider the following exact sequence induced by the Kummer exact sequence in \'etale cohomology:
    \begin{align*}
        \G_m(S)\xrightarrow{(-)^6} \G_m(S) \rightarrow \HH^1_{\mathrm{et}}(S, \mu_6) \rightarrow \Pic(S)[6]\rightarrow 0.
    \end{align*}
    The Picard group $\Pic(S)$ vanishes, being a quotient of $\Pic(\A^3_{\C})$, hence the outer term in the sequence vanishes.
    We claim that $\G_m(S)=\HH^0(S,\calO_S)^{\times} = \{ c\cdot \disc^n \mid c\in \C^{\times},\,  n  \in \Z\}$.
    Indeed, every $c\cdot \disc^n$ is clearly a unit in $\HH^0(S, \calO_S)$. 
    Conversely, given a unit $f\in \HH^0(S, \calO_S)$, seen as a rational function on $\A^3_{\C}$, its divisor $\mathrm{div}(f)$ of zeros and poles must be supported on the zero locus of $\disc$.
    Since $\disc\in \C[a,b,c]$ is irreducible, $\mathrm{div}(f) = n\cdot [\{\disc=0\}]$ for some $n\in \Z$.
    Then $\mathrm{div}(f/\disc^n) = 0$ as a rational function on $\A^3_{\C}$, hence $f/\disc^n$ is a unit in $\C[a,b,c]$, hence $f/\disc^n \in \C^{\times}$, proving the claim. 
    We conclude that the group $\HH^1_{\mathrm{et}}(S,\mu_6)$ classifying $\mu_6$-torsors is generated by the image of $\disc$.
    
    Let $\mathcal{E}_i\rightarrow S$ be the relative elliptic curve with equation $y^2 = x^3 + \disc(f)^i$.
    The previous paragraph shows that $\mathrm{J}(\mathbb{W})$ is isomorphic to $\mathcal{E}_i(\C)$ for some $i\in \{0,1,2,3,4,5\}$.
    We show that $i=1$, using our previous results on \emph{bielliptic} Picard curves in \cite{LagaShnidman-CeresaBiellipticPicard}.
    Let $T\subset S$ be the closed subscheme where $b=0$, parametrizing even quartic polynomials $f = x^4+ ax^2+c$.
    For $f\in T(\C)$, the $\mu_3$-action on $\mathcal{C}_f$ extends to a $\mu_6$-action. 
    A calculation shows $\disc|_T = 16 c (-a^2 + 4 c)^2$.
    Applying the singular cohomology realization functor to \cite[Theorem 5.1]{LagaShnidman-CeresaBiellipticPicard}, the relative complex torus $\mathrm{J}((\mathbb{V}|_{T(\C)})^{\mu_6})\rightarrow T(\C)$ is isogenous to $\mathcal{E}_1(\C)|_{T(\C)}\rightarrow T(\C)$.
    Since $\mathrm{J}((\mathbb{V}|_{T(\C)})^{\mu_6})$ is a subtorus of $\mathrm{J}((\mathbb{V}|_{T(\C)})^{\mu_3})$ of the same dimension, they must be equal, hence $\mathrm{J}(\mathbb{W})|_{T(\C)}$ is isogenous to $\mathcal{E}_1(\C)|_{T(\C)}$ over $T(\C)$.
    On the other hand, let $i\in \{0,1,2,3,4,5\}$ be such that $\mathrm{J}(\mathbb{W})\simeq \mathcal{E}_i(\C)$.
    Then $\mathrm{J}(\mathbb{W})|_{T(\C)}\simeq \mathcal{E}_i(\C)|_{T(\C)}$, hence $\mathcal{E}_1(\C)|_{T(\C)}$ is isogenous to $\mathcal{E}_i(\C)|_{T(\C)}$.
    
    We show that the latter can happen only if $i=1$.
    Indeed, let $\varphi\colon \mathcal{E}_1(\C)|_{T(\C)} \rightarrow \mathcal{E}_i(\C)|_{T(\C)}$ be an isogeny.
    Since the domain and target of $\varphi$ are isotrivial relative elliptic curves with $\calO$-multiplication, $\varphi$ factors as $\psi\circ \gamma$, where $\gamma$ is an endomorphism of $\mathcal{E}_1(\C)|_{T(\C)}$ and $\psi$ an isomorphism. 
    Therefore $\mathcal{E}_1(\C)|_{T(\C)}$ and $\mathcal{E}_i(\C)|_{T(\C)}$ are isomorphic. 
    Since the monodromy representations of $\mathcal{E}_1(\C)|_{T(\C)}$ and $\mathcal{E}_i(\C)|_{T(\C)}$ are non-isomorphic when $i\neq 1$, we conclude that $i=1$ and that $\mathrm{J}(\mathbb{W})$ is isogenous to $\mathcal{E}_1(\C)$ over $S(\C)$.
    
    In summary, we have shown that $\mathrm{J}(\mathbb{V}^{\mu_3})$, $\mathrm{J}(\mathbb{W})$ and $\mathcal{E}_1(\C)$ are isogenous. 
    Since $\mathcal{E}_1(\C)$ is isomorphic to $\mathcal{E}(\C)$, we conclude the proof.
\end{proof}

\begin{proposition}\label{proposition: group of sections of calE}
    The group of (algebraic) sections of $\mathcal{E} \rightarrow S$ is free of rank $1$ over $\calO = \Z[\omega]$, and contains the $\calO$-span of $P := (a^2+12c,72ac-2a^3-27b^2)$ as a finite index subgroup.
\end{proposition}
\begin{proof}
    Given $s = (a,b,c) \in S(\C)$, consider the closed subscheme $T = \{ (at,bt,ct^2) \colon t\in \A^1_{\C}\} \cap S$ of $S$ and the restriction $\pi\colon \mathcal{E}|_T\rightarrow T$.
    The variety $\mathcal{E}|_T$ is an open subscheme of an elliptic surface with Weierstrass equation $y^2 = x^3+t^4g(t)$, where $g(t) \in \C[t]$ has degree $\leq 2$, using the formula \eqref{equation: disc picard curve}.
    There exists a dense open $U\subset S$ such that for all $s\in U(\C)$, $g(t)$ has two distinct nonzero roots.
    Let $s\in U(\C)$. Using Tate's algorithm, we see that the elliptic surface has a singular fiber above $t=0$ with Kodaira type $\IV^*$, singular fibers above the roots of $g(t)$ with Kodaira type $\II$, and is smooth above $t= \infty$.
    The presence of a fiber of type $\II$ implies that $\mathcal{E}|_T\rightarrow T$ has no torsion sections \cite[Lemma 7.8]{SchuttShioda-ellipticsurfaces}.
    Moreover the Shioda--Tate formula \cite[Theorem 6.3, Proposition 6.6 and \S8.8]{SchuttShioda-ellipticsurfaces} shows that the group of sections of $\mathcal{E}|_T\rightarrow T$ is free of rank $2$.
    Since $\mathcal{E}|_T$ receives an $\calO$-action, its group of sections is free of rank $1$ over $\calO$.
    We will now show that these facts can be used to prove the claims of the proposition by varying $s$ in $U(\C)$.

    We first show that $\mathcal{E}(S)$ is torsion-free. 
    Suppose $Q \in \mathcal{E}(S)$ is a torsion section. 
    Then $N\cdot Q =0$ for some $N\geq 1$.
    Hence $N\cdot Q|_{T}=0$ for every $s\in U(\C)$.
    Since the group of sections of $\mathcal{E}|_T\rightarrow T$ is torsion-free, $Q|_T=0$ for every $s\in U(\C)$.
    Hence $Q|_U = 0$.
    Since $U$ is dense in $S$, we must have $Q = 0$, as desired.

    The fact that $P$ defines a section of $\mathcal{E}\rightarrow S$ can be verified by direct computation (see \S\ref{subsec: generalities}).
    Let $L = \Z\langle P, \omega\cdot P\rangle $ be the $\calO$-span of $P$, which is a subgroup of the group of sections $\mathcal{E}(S)$ of $\mathcal{E}\rightarrow S$.
    The previous paragraph shows that $L$ is free of rank $1$ over $\calO$.
    It remains to show that it has finite index in $\mathcal{E}(S)$.
    Suppose for the sake of contradiction that there exists a third section $Q \in \mathcal{E}(S)$ which is not in $L\otimes \Q$.
    Then the locus $S_{\text{dep}}$ of $s\in S(\C)$ where $P_{s}, \omega\cdot P_{s}$ and $Q_s$ are linearly dependent is a countable union of closed proper algebraic subvarieties. 
    For every $s\in U(\C)$, the group of sections of $\mathcal{E}|_T\rightarrow T$ is free of rank $1$ over $\calO$.
    There exists a possibly smaller dense open $V\subset U$ such that if $s\in V(\C)$ then $P_{s}$ is nonzero, hence $\langle P_{s}, \omega \cdot P_s \rangle$ is a finite index subgroup of $\mathcal{E}(T)$.
    Therefore $Q_t, P_{1,t}, P_{2,t}$ are linearly dependent for all $s\in V(\C)$.
    Since $V(\C)\setminus S_{\text{dep}}$ is nonempty, we obtain a contradiction.
\end{proof}

\subsection{Proof of the vanishing criterion in the Chow group}\label{subsec: proof of the vanishing criterion picard chow group}

\begin{proof}[Proof of Theorem \ref{thm: picard vanishing chow group}]
    We may assume (using Lemma \ref{lemma: base extension chow groups}) that $k=\C$.
    Recall that $3\cdot \AJ(\kappa_{\mathcal{C}})$ defines a holomorphic section of $\mathrm{J}(\mathbb{V}^{\mu_3})\rightarrow S(\C)$.
    Choose an isogeny of complex tori $\mathrm{J}(\mathbb{V}^{\mu_3})\rightarrow \mathcal{E}(\C)$ using Proposition \ref{proposition: identification complex tori} and let $\sigma$ be the image of $3\cdot \AJ(\kappa_{\mathcal{C}})$ under this isogeny.
    This is a holomorphic section of $\mathcal{E}(\C) \rightarrow S(\C)$.

    We claim that $\sigma$ is not a torsion section. 
    If it were torsion, then $\AJ(\kappa_{\mathcal{C}})$ would be a torsion section of $\mathrm{J}(\mathbb{V}^{\mu_3})$, hence, by Proposition \ref{proposition: ceresa torsion iff AJ image torsion}, $\kappa_f$ would be torsion for every $f\in S(\C)$.
    This is not the case, since $\kappa_f$ is of infinite order if $f = x^4+x^2+1$ by \cite[Corollary 2.9]{LagaShnidman-CeresaBiellipticPicard}.
    We conclude that $\sigma$ is not a torsion section.

    Next we claim that $\sigma$ is the analytification of an algebraic section of $\mathcal{E}\rightarrow S$.
    Let $N\geq 1$ be an integer such that $N\cdot \kappa_f$ is algebraically trivial for every Picard curve $f\in S(\C)$ (such an integer exists by Theorem \ref{thrm main body: picard curve vanish griffiths group}).
    By the algebraicity of the Abel--Jacobi map for algebraically trivial cycles (\cite[Theorem 1]{achteretal-normalfunctionsalgebraic}), there exists a relative algebraic subtorus $\mathrm{J}_a(\mathbb{V})\subset \mathrm{J}(\mathbb{V})$ with the following property: the section $\mathrm{AJ}(3N\cdot  \kappa_{\mathcal{C}})$ lands in $\mathrm{J}_a(\mathbb{V})$ and the corresponding holomorphic map $S(\C) \rightarrow \mathrm{J}_a(\mathbb{V})$ is algebraic.
    On the other hand, $\AJ(3N\cdot \kappa_{\mathcal{C}})$ also lands in $\mathrm{J}(\mathbb{V}^{\mu_3})$, is not a torsion section by the previous paragraph, and $\mathrm{J}(\mathbb{V}^{\mu_3})$ has relative dimension $1$ over $S(\C)$.
    Therefore $\mathrm{J}(\mathbb{V}^{\mu_3})$ is the smallest relative subtorus of $\mathrm{J}(\mathbb{V})$ containing the image of $\AJ(3N \cdot \kappa_{\mathcal{C}})$.
    Hence $\mathrm{J}(\mathbb{V}^{\mu_3})\subset \mathrm{J}_a(\mathbb{V})$.
    We conclude that $\AJ(3N \cdot\kappa_{\mathcal{C}}) \colon S(\C) \rightarrow \mathrm{J}(\mathbb{V}^{\mu_3})$ is algebraic, so $\AJ(3\cdot \kappa_{\mathcal{C}})$ is algebraic, so $\sigma$ is algebraic.

    Since $\sigma$ is an algebraic section of $\mathcal{E}\rightarrow S$, Proposition \ref{proposition: group of sections of calE} shows that there exists an integer $M\geq 1$ and an element $\gamma \in \calO$ such that $M \cdot \sigma = \gamma \cdot P$.
    Since $\sigma$ is not a torsion section, $\gamma \neq 0$.
    
    Putting everything together, we have for $f\in S(\C)$: $\kappa_f$ is torsion if and only if $\AJ_{J_f}^2(\kappa_f)$ torsion (by Proposition \ref{proposition: ceresa torsion iff AJ image torsion}), if and only if $\AJ(\kappa_{\mathcal{C}})_f \in \mathrm{J}(\mathbb{V}^{\mu_3})_f$ torsion, if and only if $\sigma_f \in \mathcal{E}_f(\C)$ torsion, if and only if $P_f$ torsion.
    Tracing through the equivalences, the quotient of the torsion orders $\ord(\kappa_f)/\ord(P_f)$ (if defined) takes only finitely many values as $f$ ranges in $S(\C)$.
\end{proof}

\subsection{Analyzing the torsion locus of $P$}

Theorem \ref{thm: picard vanishing chow group} shows that there are infinitely many plane quartic curves over $\Q$ with torsion Ceresa cycle, since we may take $a = c = -12$. In fact, we can find explicit families of torsion Ceresa cycles over $\overline{\Q}$ of arbitrarily large order by explicitly computing the torsion locus of the section $P$ of $\mathcal{E}\rightarrow S$.

Recall that $S\subset \A^3_{\C}$ is the parameter space of polynomials $f(x) = x^4 + ax^2 + bx+c$ of nonzero discriminant.
For $(I,J) \in \mathbb{C}^2$ such that $4I^3- J^2 \neq 0$, let $X_{I,J}\subset S$ be the closed subscheme of elements $f$ satisfying $(I(f), J(f)) = (I,J)$, and let $E_{I,J}$ be the elliptic curve $y^2 = x^3- Ix /3 - J /27$.
Let $E_{I,J}^{\circ}\subset E_{I,J}$ be the complement of the origin.
\begin{lemma}\label{lemma: bijection polys fixed invariants and E_{I,J}}
For $(I,J)\in \C^2$ such that $4I^3 -J^2 \neq 0$, the map $f\mapsto (-2a/3,b)$ is an isomorphism $X_{I,J} \simeq E_{I,J}^{\circ}$.
\end{lemma}
\begin{proof}
Indeed, a calculation shows that an inverse is given by 
\begin{align}\label{equation: quartic ass to point on E}
    (\alpha, \beta)\mapsto f(x) = x^4 - 3\alpha x^2/2 + \beta x + (I/12 - 3\alpha^2/16).
\end{align}
\end{proof}


We now describe the locus of $f\in S$ where the specialization $P_f = (I(f), J(f))$ of the section $P$ of $\mathcal{E}\rightarrow S$ is torsion.
First we consider the restriction of this locus to the closed subscheme $S_1\subset S$ where $\disc(f) =1$, so that $\mathcal{E}\rightarrow S$ restricts to the trivial family with fiber the elliptic curve $E_0\colon y^2 = 4x^3-27$. 
Let $E_0^{\circ}\subset E_0$ be the complement of the origin.
If $(I,J) \in E_0^{\circ}(\C)$, then $X_{I,J}$ is a closed subscheme of $S_1$. 
For $G$ a group, write $G_{tors}\subset G$ for the subset of torsion elements.
Then we tautologically have
\begin{align*}
\{f\in S_1(k) \colon P_f \in E_0(\C)_{tors}\}
=
\bigsqcup_{(I,J) \in E_0(\C)_{tors}} X_{I,J}(\C).
\end{align*}

Define a $\G_m$-action on $S$ via the formula $\lambda \cdot (a,b,c) = (\lambda^2a,\lambda^3 b, \lambda^4 c)$.
The polynomials $I(f), J(f)$ and $\disc(f)$ are then homogenous of degree $4,6,12$.
Using the $\G_m$-action on $S$, we can form the (coarse) quotient $\mathcal{M}_{Pic} = S/\G_m$, an open subscheme of weighted projective space $\mathbb{P}(2,3,4)$.
\begin{lemma}
The variety $\mathcal{M}_{Pic}$ is the coarse moduli space of Picard curves over $\C$.
\end{lemma}
\begin{proof}
See \cite[p 15, Proposition 2.3]{Holzapfel-theballandsomehilbert}.
\end{proof}
The inclusion $S_1\subset S$ determines an isomorphism $S_1/\mu_{12}\simeq \mathcal{M}_{Pic}$, where $\mu_{12}\subset \G_m$ acts via restriction on $S_1$.
Write $\pi\colon S\rightarrow \mathcal{M}_{Pic}$ and $\pi_1\colon S_1\rightarrow \mathcal{M}_{Pic}$ for the quotient maps.

Let $V\subset \mathcal{M}_{Pic}$ be the image of the torsion locus in $S$ of the section $P$ of $\mathcal{E}\rightarrow S$.
Since this torsion locus in $\G_m$-invariant, $V$ also equals the image of the torsion locus of $P_f \in E_0$ under $\pi_1$.
Therefore $V$ equals $\cup \pi_1(X_{I,J})$, where $(I,J)$ ranges over torsion points of $E_{0}^{\circ}$. 
If $\zeta\in \mu_{12}$, then $\zeta\cdot X_{I,J} = X_{\zeta^4 I , \zeta^6 J}$, so $\mu_2\subset \mu_{12}$ preserves $X_{I,J}$.
Under the bijection $X_{I,J}\simeq E_{I,J}^{\circ}$, $-1$ acts as $(\alpha, \beta)\mapsto (\alpha, -\beta)$ on $E_{I,J}^{\circ}$.
Therefore the quotient $X_{I,J}/\mu_2$ is isomorphic, via the $a$-coordinate, to $\A^1$.
It follows that for each torsion point $(I,J)$, the restriction $X_{I,J}\hookrightarrow S_1 \rightarrow \mathcal{M}_{Pic}$ factors through a map $\varphi_{(I,J)}\colon X_{I,J}/\mu_2 =\A^1 \rightarrow \mathcal{M}_{Pic}$, and $V$ is the union of the images of $\varphi_{(I,J)}$. 
If $(I,J) = (\zeta^4 I',\zeta^6 J')$ for some $\zeta \in \mu_{12}$, then the images of $\varphi_{(I,J)}$ are equal.
Otherwise, they are disjoint.
We conclude:

\begin{proposition}\label{proposition: vanishing locus union of rational curves}
We have
\begin{align*}
V = \bigcup_{(I,J) \in E_0(\mathbb{C})_{tors}} \text{Image}(\varphi_{(I,J)}\colon \A^1 \rightarrow \mathcal{M}_{Pic}).
\end{align*}
The images of two maps $\varphi_{(I,J)}$ are either equal or disjoint.
The maps $\varphi_{(I,J)}$ are quasi-finite.
\end{proposition}

We can identify $\varphi_{(I,J)}(t)$ explicitly, at least when $b\neq 0$. 
For $t\in \C$, let $g_{(I,J)}(t)= t^3 - It/3 -J/27$, let $(\alpha_t, \beta_t) = (t g_{(I,J)}(t), g_{(I,J)}(t)^2)$ and let $f_{(I,J),t}(x) = x^4 - 3\alpha_t x^2/2 + \beta_t x + (I/12 - 3\alpha_t^2/16)$.

\begin{proposition}\label{proposition: explicit description rational curves vanishing locus}
In the above notation, assume $g_{(I,J)}(t)\neq 0$.
Then $\varphi_{(I,J)}(t) = \pi(f_{(I,J),t})$.
\end{proposition}
\begin{proof}
Let $\tilde{f}$ be an element of $X_{I,J}$ which under the isomorphism $X_{I,J}\simeq E_{I,J}^{\circ}$ corresponds to an element $(\tilde{\alpha}, \tilde{\beta})$ with $\tilde{\alpha} = t$. 
By definition, $\varphi_{(I,J)}(t) = \pi_1(\tilde{f})$. 
We have $\tilde{\beta}^2 = g_{(I,J)}(t)\neq 0$. 
A calculation shows that $f_{(I,J),t} = \tilde{\beta}\cdot \tilde{f}$, so $f_{(I,J), t}$ and $\tilde{f}$ have the same image in $\mathcal{M}_{Pic}$, as claimed.
\end{proof}

Proposition \ref{proposition: explicit description rational curves vanishing locus} gives explicit one-parameter families of Picard curves with equation $C_t\colon y^3 = f_{(I,J),t}(x)$ whose Ceresa cycle is torsion. The explicit description shows that when $(I,J)\in k^2$, where $k$ is a subfield of $\mathbb{C}$, then the family $C_t$ is defined over $k$.
Since all torsion points of $E_0$ are defined over $\Qbar$, so are the families $C_t$.

We can now prove Theorem \ref{thm: corollaries}.

\begin{proof}[Proof of Theorem \ref{thm: corollaries}]
\begin{enumerate}[$(i)$]
\item This follows from Theorems \ref{thm: picard vanishing chow group} and Proposition \ref{proposition: vanishing locus union of rational curves}.
\item For every $d \geq 1$, there exists $N = N(d) \geq 1$ such that for every Picard curve $C_f$ over a number field $k$ of degree $d$, the order of $\kappa_f$ in $\CH_1(J_{f,\bar{k}})$ is either infinite or less than $N$.  
Indeed, this follows from Theorem \ref{thm: picard vanishing chow group} and the uniform bound (depending only on $d$) on the order of a $k$-rational torsion point on an elliptic curve $y^2 = x^3 + D$ over any number field $k$ of degree $d$.  
\end{enumerate}
\end{proof}

\section{Ceresa vanishing loci in genus $3$}\label{sec: automorphism strata and Ceresa vanishing genus 3}

Fix $g\geq 3$ and let $\mathcal{M}_g$ be the (coarse) moduli space of genus-$g$ curves, seen as a variety over $\Q$.
Let $V_g^{\mathrm{rat}}\subset \mathcal{M}_g$ be the subset of curves $[C]$ for which $\kappa(C)$ vanishes in $\CH_1(\Jac(C))_{\Q}$, in the notation of \S\ref{subsec: ceresa cycles}.
Since the vanishing of $\kappa(C)$ only depends on the geometric isomorphism class of $C$ (by Lemma \ref{lemma: base extension chow groups}), this locus is well defined.
Similarly define the locus $V_g^{\mathrm{alg}} \subset \mathcal{M}_g$ where $\bar{\kappa}(C) \in \Gr_1(\Jac(C))_{\Q}$ vanishes.
\begin{lemma}\label{lemma: Zg,barZg countable union}
    The subsets $V_g^{\mathrm{rat}}, V_g^{\mathrm{alg}}\subset \mathcal{M}_g$ are countable unions of proper closed algebraic subvarieties.
\end{lemma}
\begin{proof}
    Let $\widetilde{\mathcal{M}}_g$ be the fine moduli space parametrizing genus-$g$ curves $C$ with full symplectic level-$5$ structure and a degree-$1$ divisor class $e\in \CH_0(C)$ such that $(2g-2)e$ is canonical.
    Considering the universal curve over it together with its degree-$1$ divisor, we can define a universal Ceresa cycle on the universal Jacobian over $\widetilde{\mathcal{M}}_g$; Lemma \ref{lemma: locus of rational triviality countable union} then implies that the locus in $\widetilde{\mathcal{M}}_g$ where this Ceresa cycle vanishes (with $\Q$-coefficients) is a countable union of closed algebraic subvarieties.
    Since the forgetful map $\widetilde{\mathcal{M}}_g\rightarrow \mathcal{M}_g$ is proper, the same is true for the image of this locus, which is exactly $V_g^{\mathrm{rat}}$. 
    The proof for $V_g^{\mathrm{alg}}$ is identical.
 \end{proof}

These vanishing loci have the following basic properties: $V_g^{\mathrm{rat}}\subset V_g^{\mathrm{alg}}\subset \mathcal{M}_g$; the hyperelliptic locus is contained in $V_g^{\mathrm{rat}}$; and $V_g^{\mathrm{alg}} \neq \mathcal{M}_g$ by Ceresa's famous result \cite{Ceresa}.
It would be interesting to obtain further information about the components of $V_g^{\mathrm{rat}}$ and $V_g^{\mathrm{alg}}$. 
We end our paper by determining the automorphism group strata in $\mathcal{M}_3$ that are contained in $V_3^{\mathrm{rat}}$ or $V_3^{\mathrm{alg}}$. 
So let $g=3$ and consider the open subscheme $\mathcal{M}_3^{\mathrm{nh}}\subset \mathcal{M}_3$ of non-hyperelliptic curves.
There is a stratification $\mathcal{M}_3^{\mathrm{nh}} = \sqcup X_G$ into locally closed subvarieties such that a non-hyperelliptic curve $C$ over $\C$ belongs to $X_G$ if and only if $\Aut(C) \simeq G$.
It turns out that $X_G$ is irreducible and the closure of $X_G$ is a union of other strata. 
We refer to \cite[\S2.2]{lombardolorenzogarciaritzenthaler} and references therein for a complete description of the loci $X_G$ and the closure relations between them. 
We reproduce here a diagram capturing these closure relations:

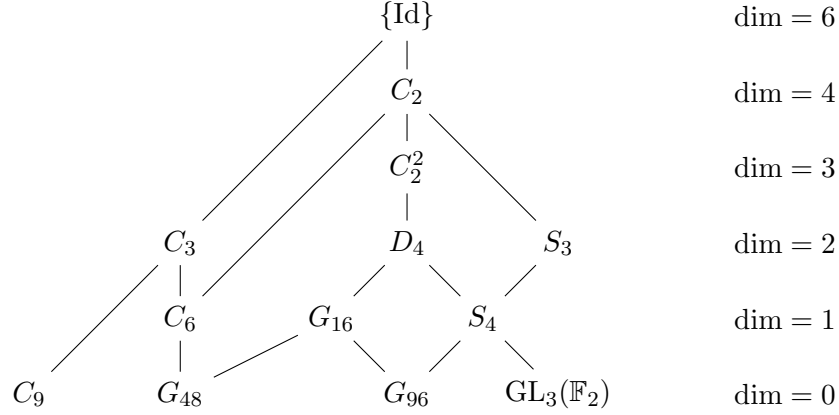
\begin{center}
\begin{tikzpicture}
    \node (c1) at (0,3) {$\{\operatorname{Id}\}$};
    \node (c2) at (0,2) {$C_2$};
    \node (c2c2) at (0,1) {$C_2^2$};
    \node (c3) at (-3,0) {$C_3$};
    \node (d4) at (0,0) {$D_4$};
    \node (s3) at (2,0) {$S_3$};
    \node (c6) at (-3,-1) {$C_6$};
    \node (g16) at (-1,-1) {$G_{16}$};
    \node (s4) at (1,-1) {$S_4$};
    \node (c9) at (-5,-2) {$C_9$};
    \node (g48) at (-3,-2) {$G_{48}$};
    \node (g96) at (0,-2) {$G_{96}$};
    \node (g168) at (2,-2) {$\GL_3(\F_2)$};
    \node (dim6) at (5,3) {$\dim = 6$};
    \node (dim4) at (5,2) {$\dim = 4$};
    \node (dim3) at (5,1) {$\dim = 3$};
    \node (dim2) at (5,0) {$\dim = 2$};
    \node (dim1) at (5,-1) {$\dim = 1$};
    \node (dim0) at (5,-2) {$\dim = 0$};
    \draw (c1) -- (c2)--(c2c2)--(d4)--(g16)--(g96);
    \draw (c1)--(c3)--(c9);
    \draw (c3)--(c6)--(g48);
    \draw (g16)--(g48);
    \draw (c2)--(c6);
    \draw (c2)--(s3)--(s4)--(g168);
    \draw (d4)--(s4)--(g96);\end{tikzpicture}
\end{center}
For $n\in \{16,48, 96\}$, the symbol $G_n$ the group of order $n$ and GAP label $(16,13), (48, 33),$ and $(96,64)$ respectively \cite{GAP4}.
See \cite[Table 2]{lombardolorenzogarciaritzenthaler} for models for a generic plane quartic in $X_G$.
We make explicit the strata that are relevant for us: $\overline{X_{C_3}}$ is the locus of Picard curves studied in \S\ref{sec: Picard}; $\overline{X_{C_6}}$ is the locus of \emph{bielliptic} Picard curves studied in \cite{LagaShnidman-BiellipticPicardCurves}; and the zero-dimensional strata each consist of a single automorphism-maximal curve with equation
\begin{align*}
\begin{cases}
    y^3z = x^4+xz^3 & \text{ if }G = C_9, \\
    y^3z = x^4+z^4 & \text{ if }G = G_{48,}\\
    x^4+y^4+z^4=0 & \text{ if }G = G_{96},\\
    x^3y+y^3z+z^3x & \text{ if }G = \GL_3(\F_2).
\end{cases}
\end{align*}
If $C$ is a non-hyperelliptic genus $3$ curve over a field $k$, we say $C$ is a generic curve for $X_G$ if the classifying map $\Spec(k) \rightarrow \mathcal{M}_3$ maps to the generic point of $X_G$.
If $X$ is an integral variety over $\C$, we say a property hold for a very general $x\in X(\C)$ if it holds true outside a countable union of proper closed subvarieties of $X$.

\begin{lemma}\label{lemma: equivalent conditions strata identically vanishing}
    For a group $G$ in the above diagram, the following are equivalent:
    \begin{enumerate}
        \item $\kappa(C)\neq 0$ for some generic curve $C$ for $X_G$.
        \item $\kappa(C)\neq 0$ for a very general $C$ in $X_G(\C)$.
        \item $\kappa(C)\neq 0$ for some $C$ in $X_G$.
        \item $X_G\not \subset V_3^{\mathrm{rat}}$.
        \item $\overline{X_G}\not\subset V_3^{\mathrm{rat}}$.
    \end{enumerate}
    Moreover, the analogous equivalences hold for $\bar{\kappa}(C)\in \Gr_1(J)_\Q$ and $V_3^{\mathrm{alg}}$.
\end{lemma}
\begin{proof}
    Follows immediately from Lemma \ref{lemma: Zg,barZg countable union}.
\end{proof}

We end with the proof of Theorem \ref{theorem: main theorem E}, which we restate for convenience:

\begin{proposition}
    Let $G$ be a group in the diagram.
    Then 
    \begin{enumerate}
        \item $X_G\subset V_3^{\mathrm{alg}}$ if and only if $G=C_3,C_6,C_9, G_{48}$.
        \item $X_G \subset V_3^{\mathrm{rat}}$ if and only if $G = C_9, G_{48}$;
    \end{enumerate}
\end{proposition}
\begin{proof}
    \begin{enumerate}
        \item The equivalence between $(4)$ and $(5)$ of Lemma \ref{lemma: equivalent conditions strata identically vanishing} implies that if $X_H\subset \overline{X_G}$ and $X_G\subset V_3^{\mathrm{rat}}$, then $X_H\subset V_3^{\mathrm{rat}}$; we will use this observation in the remainder of the proof.

        Our analysis of Picard curves (Theorem \ref{thm: Picard main}) shows that $X_{C_3}\subset V_3^{\mathrm{rat}}$, so $X_G\subset V_3^{\mathrm{rat}}$ for $G= C_6, C_9$ and $G_{48}$ as well.
        On the other hand, the Ceresa cycle of the Fermat quartic and Klein quartic are known to be of infinite order in the Griffiths group; see \cite[Theorem (4.1)]{BlochCrelleI} for the former and \cite[\S4]{Kimura-modifieddiagonalfermat} for the latter.
        By our observation, this means that $X_G\not\subset V_3^{\mathrm{rat}}$ for every stratum whose closure contains one of these curves. 
        Since every stratum not contained in $\overline{X_{C_3}}$ has this property, we conclude the proof.
        \item Since $V_3^{\mathrm{rat}}\subset V_3^{\mathrm{alg}}$, Part $(1)$ implies that $X_G\subset V_3^{\mathrm{rat}}$ only if $G = C_3,C_6,C_9$ or $G_{48}$.
        The criterion of Theorem \ref{thm: main} applies to the curves in $X_{C_9}$ and $X_{G_{48}}$ (see Example \ref{example: theorem A genus 3 examples}), so $V_3^{\mathrm{rat}}$ contains these strata.
        On the other hand, there exist curves in $X_{C_6}$ with nonvanishing $\kappa(C)$, by \cite[Corollary 1.2]{LagaShnidman-CeresaBiellipticPicard}.
        So $X_{C_6}\not\subset V_3^{\mathrm{rat}}$ and $X_{C_3}\not\subset V_3^{\mathrm{rat}}$.
    \end{enumerate}
\end{proof}



\bibliographystyle{abbrv}

\end{document}